\documentclass[11pt, reqno]{amsart}
\usepackage{amssymb,latexsym,amsmath,amsfonts,mathdots,enumitem}
\usepackage{latexsym}
\usepackage[mathscr]{eucal}
\usepackage{colortbl,xcolor}
\usepackage{lmodern}
\usepackage{sansmathaccent}
\usepackage[latin1]{inputenc}
\usepackage{tikz}
\usepackage{physics}
\usetikzlibrary{shapes,arrows}
\usetikzlibrary{matrix,calc,shapes,arrows,positioning}
\pdfmapfile{+sansmathaccent.map}

\numberwithin{equation}{section}
\theoremstyle{plain}
\newtheorem{exam}{Example}

\newtheorem{thm}{Theorem}[section]

\newtheorem{cor}[thm]{Corollary}
\newtheorem{lem}[thm]{Lemma}
\newtheorem{prop}[thm]{Proposition}
\theoremstyle{definition}
\newtheorem{defn}[thm]{Definition}
\newtheorem{rem}[thm]{Remark}

\numberwithin{equation}{section}
\def\A{{\mathcal A}}

\def\beq{\begin{eqnarray}}
	\def\eeq{\end{eqnarray}}
\def\beqa{\begin{eqnarray*}}
	\def\eeqa{\end{eqnarray*}}

\def\Ran{\operatorname{Ran}}

\def\beqn{\begin{equation}}
	\def\eeqn{\end{equation}}

\def\mg#1{}

\def\Ran{\operatorname{Ran}}

\renewcommand{\epsilon}{\varepsilon}
\renewcommand{\phi}{\varphi}

\renewcommand{\bf}[1]{\textbf{#1}}
\renewcommand{\it}[1]{\textit{#1}}
\renewcommand{\sc}[1]{\textsc{#1}}
\renewcommand{\sf}[1]{\textsf{#1}}

\numberwithin{equation}{section}
\allowdisplaybreaks[4] 

\setlist[enumerate]{font=\upshape,noitemsep, topsep=0pt} 
\setlist[itemize]{noitemsep, topsep=0pt}

\begin{document}
	
	\title[Admissible Fundamental Operators and Models for $\Gamma_{E(3; 3; 1, 1, 1)}$ and $\Gamma_{E(3; 2; 1, 2)}$-contraction]{Admissible Fundamental Operators and Models for $\Gamma_{E(3; 3; 1, 1, 1)}$-contraction and $\Gamma_{E(3; 2; 1, 2)}$-contraction}
	\author{Avijit Pal ~ \hspace{0.1cm} and ~ \hspace{0.1cm} Bhaskar Paul}
	\address[A. Pal]{Department of Mathematics, IIT Bhilai, 6th Lane Road, Jevra, Chhattisgarh 491002}
\email{A. Pal:avijit@iitbhilai.ac.in}

\address[B. Paul]{Department of Mathematics, IIT Bhilai, 6th Lane Road, Jevra, Chhattisgarh 491002}
\email{B. Paul:bhaskarpaul@iitbhilai.ac.in }



	\maketitle
	
	\begin{abstract}
		A $7$-tuple of commuting bounded operators $\textbf{T} = (T_1, \dots, T_7)$ defined on a Hilbert space $\mathcal{H}$ is called a \textit{$\Gamma_{E(3; 3; 1, 1, 1)}$-contraction} if $\Gamma_{E(3; 3; 1, 1, 1)}$ is a spectral set for $\textbf{T}$. Let $(S_1, S_2, S_3)$ and $(\tilde{S}_1, \tilde{S}_2)$ represents  tuples of commuting bounded operators on a Hilbert space $\mathcal{H}$ with $S_i\tilde{S}_j = \tilde{S}_jS_i$ for $1 \leqslant i \leqslant 3$ and $1 \leqslant j \leqslant 2$. The tuple $\textbf{S} = (S_1, S_2, S_3, \tilde{S}_1, \tilde{S}_2)$ is said to be  \textit{$\Gamma_{E(3; 2; 1, 2)}$-contraction} if $ \Gamma_{E(3; 2; 1, 2)}$ is a spectral set for $\textbf{S}$.	
		
		In this paper, we show that for a given pure contraction $T_7$ acting on a Hilbert space $\mathcal{H}$, if $(\tilde{F}_1, \dots, \tilde{F}_6) \in \mathcal{B}(\mathcal{D}_{T^*_7})$ with $[\tilde{F}_i, \tilde{F}_j] = 0, [\tilde{F}^*_i, \tilde{F}_{7-j}] = [\tilde{F}^*_j, \tilde{F}_{7-i}]$,$w(\tilde{F}^*_i + \tilde{F}_{7-i}z) \leqslant 1$ and these operators satisfy
		\[(\tilde{F}^*_i + \tilde{F}_{7-i}z)\Theta_{T_7}(z) = \Theta_{T_7}(z)(F_i + F^*_{7-i}z) \,\, \text{for all} \,\, z \in \mathbb{D}\] for $1 \leqslant i, j \leqslant 6$ for some $(F_1, \dots, F_6) \in \mathcal{B}(\mathcal{D}_{T_7})$ with $w(F^*_i + F_{7-i}z) \leqslant 1$ for $1 \leqslant i \leqslant 6$, then there exists a $\Gamma_{E(3; 3; 1, 1, 1)}$-contraction $(T_1, \dots, T_7)$ such that $F_1, \dots, F_6$ are the fundamental operators of $(T_1, \dots, T_7)$ and $\tilde{F}_1, \dots, \tilde{F}_6$ are the fundamental operators of $(T^*_1, \dots, T^*_7)$. We also prove similar type of result for pure $\Gamma_{E(3; 2; 1, 2)}$-contraction.
		
		We explicitly construct a $\Gamma_{E(3; 3; 1, 1, 1)}$-unitary (respectively, a $\Gamma_{E(3; 2; 1, 2)}$-unitary) starting from a $\Gamma_{E(3; 3; 1, 1, 1)}$-contraction (respectively, a $\Gamma_{E(3; 2; 1, 2)}$-contraction).
Further, we develop functional models for general $\Gamma_{E(3; 3; 1, 1, 1)}$-isometries (respectively, $\Gamma_{E(3; 2; 1, 2)}$-isometries). In particular, we construct Douglas-type and Sz.-Nazy-Foias-type models for $\Gamma_{E(3; 3; 1, 1, 1)}$-contractions (respectively, $\Gamma_{E(3; 2; 1, 2)}$-contractions). Finally, we present a Schaffer-type model for the $\Gamma_{E(3; 3; 1, 1, 1)}$-isometric dilation (respectively, the $\Gamma_{E(3; 2; 1, 2)}$-isometric dilation).
	\end{abstract}
	
	\section{Introduction and Motivation}\label{Intro}
	
	Let $\Omega$ be a compact subset of $\mathbb{C}^m$, and let $\mathcal{A}(\Omega)$ denote the algebra of holomorphic functions defined on an open set $U$ containing $\Omega$. Consider an $m$-tuple of commuting bounded operators $\mathbf{T} = (T_1, \ldots, T_m)$ acting on a Hilbert space $\mathcal{H}$, and let $\sigma(\mathbf{T})$ denote its joint spectrum.
We define a homomorphism $\rho_{\mathbf T}:\mathcal A(\Omega)\rightarrow\mathcal B(\mathcal H)$ in a following manner: $$1\to I~{\rm{and}}~z_i\to T_i~{\rm{for}}~1\leq i\leq m. $$ It is evident that $\rho_{\mathbf T}$ is a homomorphism. A compact set $\Omega \subset \mathbb{C}^m$ is called a spectral set for $\mathbf{T}$ if $\sigma(\mathbf{T}) \subseteq \Omega$ and the homomorphism $\rho_{\mathbf{T}}$ is contractive. Von Neumann introduced this notion in the one-variable case. In particular, his classical theorem asserts that the closed unit disc is a spectral set for every contraction on a Hilbert space $\mathcal{H}$.
	\begin{thm}[Chapter 1, Corollary 1.2, \cite{paulsen}]
		Let $T\in \mathcal B(\mathcal H)$ be a contraction. Then
		$$\|p(T)\|\leq \|p\|_{\infty, \bar{\mathbb D}}:=\sup\{|p(z)|: |z|\leq1\} $$ for every polynomial $p.$
	\end{thm}
	The following theorem presented here is a revised version of the Sz.-Nagy dilation theorem [Theorem 1.1, \cite{paulsen}].
	\begin{thm}[Theoremn $4.3$, \cite{paulsen}] Let $T\in \mathcal B(\mathcal H)$ be a contraction. Then there exists a larger Hilbert space $\mathcal K$ that contains $\mathcal H$ as a subspace, and a unitary operator $U$ acting on a Hilbert space $\mathcal K \supseteq \mathcal H$ with the property that $\mathcal K$ is the smallest closed reducing subspace for $U$ containing $\mathcal H$ such that
		$$P_\mathcal H\,U^n_{|\mathcal H}=T^n, ~{\rm{ for ~all}} ~n\in \mathbb N\cup \{0\}.$$
	\end{thm}
	
	Schaffer constructed the unitary dilation for a given contraction $T$. The spectral theorem for unitary operators then guarantees the von Neumann inequality through the existence of a power dilation. Let $\Omega$ be a compact subset of $\mathbb C ^m. $ Let $F=\left(\!(f_{ij})\!\right)$ be a matrix-valued polynomial defined on $\Omega.$ We define $\Omega$ as a complete spectral set (complete $\Omega$-contraction) for $\mathbf T$ if $\|F(\mathbf T) \| \leq \|F\|_{\infty, \Omega}$ for every $F\in \mathcal O(\Omega)\otimes \mathcal M_{k\times k}(\mathbb C), k\geq 1$.If a compact set $\Omega$ serves as a spectral set for a commuting $m$-tuple of operators $\mathbf{T}$, then $\Omega$ is, in fact, a complete spectral set for $\mathbf{T}$. In this case, we say that the domain $\Omega$ possesses property $P$. We say that a $m$-tuple of  commuting bounded operators $\mathbf{T}$ with $\Omega$ as a spectral set has a $\partial \Omega$ normal dilation if there is a Hilbert space $\mathcal K$ that contains $\mathcal H$ as a subspace, along with a commuting $m$-tuple of normal operators $\mathbf{N}=(N_1,\ldots,N_m)$ on $\mathcal K$ whose spectrum lies within $\partial \Omega$, and$$P_{\mathcal H}F(\mathbf N)\mid_{\mathcal H}=F(\mathbf T) ~{\rm{for~ all~}} F\in \mathcal O(\Omega).$$
	
	In 1969, Arveson \cite{A,AW} established that a commuting $m$-tuple of operators $\mathbf{T}$ has a $\partial \Omega$-normal dilation if and only if $\Omega$ is a spectral set for $\mathbf{T}$ and $\mathbf{T}$ satisfies property $P$. Later, Agler \cite{agler} proved in 1984 that the annulus possesses property $P$. However, Dritschel and McCullough \cite{michel} showed that property $P$ fails for domains with connectivity $n \ge 2$. In several complex variables, both the symmetrized bidisc and the bidisc are known to possess property $P$, as shown by Agler and Young \cite{young} and Ando \cite{paulsen}, respectively. Parrott \cite{paulsen} provided the first counterexample in the multivariable setting for the polydisc $\mathbb{D}^n$ when $n > 2$. Subsequently, G. Misra \cite{GM,sastry}, V. Paulsen \cite{vern}, and E. Ricard \cite{pisier} established that no ball in $\mathbb{C}^m$, defined with respect to any norm $|\cdot|_{\Omega}$ and for $m \ge 3$, possesses property $P$. Furthermore, \cite{cv} shows that if two matrices $B_1$ and $B_2$ are not simultaneously diagonalizable via a unitary transformation, then the set  $\Omega_{\mathbf B}:= \{(z_1,z_2) :\|z_1 B_1 + z_2 B_2 \|_{\rm op} < 1\}$ fails to have property $P$, where $\mathbf{B} = (B_1, B_2) \in \mathbb{C}^2 \otimes \mathcal{M}_2(\mathbb{C})$ and $B_1, B_2$ are linearly independent.	
	
	We recall the definition of completely non-unitary contraction from \cite{Nagy}. A contraction $T$  on a Hilbert space $\mathcal H$  is said to be  completely non-unitary (c.n.u.) contractions if there exists no nontrivial reducing subspace $\mathcal L$  for $T$ such that $T |_{\mathcal L}$ is a unitary operator. This section presents the canonical decomposition of the $\Gamma_{E(3; 3; 1, 1, 1)}$-contraction and the $\Gamma_{E(3; 2; 1, 2)}$-contraction.  Any contraction $T$ on a Hilbert space $\mathcal{H}$ can be expressed as the orthogonal direct sum of a unitary and a completely non-unitary contraction. The details can be found in [Theorem 3.2, \cite{Nagy}]. We start with the following definition, which will be essential for the canonical decomposition of the $\Gamma_{E(3; 3; 1, 1, 1)}$-contraction and the $\Gamma_{E(3; 2; 1, 2)}$-contraction.
	
	Let us recall \textit{spectrum, spectral radius, numerical radius} of a bounded operator $T$. The spectrum $\sigma(T)$ of $T$ is defined by
	\[\sigma(T) = \{\lambda \in \mathbb{C} : T - \lambda I \,\, \text{is not invertible}\}.\]
	The spectral radius of $T$ is denoted by $r(T)$  and defined by
	\[r(T) = \sup_{\lambda \in \sigma(T)} |\lambda|.\]
	In addition to it, the numerical radius $w(T)$ of $T$ is defined by
	\[w(T) = \sup_{||x|| \leqslant 1} |\langle Tx, x \rangle|.\]
	By some routine computation we can show that
	\[r(T) \leqslant w(T) \leqslant ||T|| \leqslant 2w(T).\]
		
	Let $T$ be a contraction a Hilbert space $\mathcal{H}$. The \textit{defect operator} of $T$ is defined by $D_T = (I - T^*T)^{1/2}$ and the \textit{defect space} of $T$ is defined by $\mathcal{D}_T = \overline{\Ran}D_T$. It follows from \cite{Nagy} that $D_T$ and $D_{T^*}$ satisfy the following identities:
	\begin{equation*}
		\begin{aligned}
			TD_T &= D_{T^*}T.
		\end{aligned}
	\end{equation*}	
	The \textit{characteristic function} $\Theta_T$ of $T$ is defined as follows:
	\begin{equation}\label{Characteristic}
		\begin{aligned}
			\Theta_T(z)
			&= (- T + D_{T^*}(I - zT^*)^{-1}D_T)|_{\mathcal{D}_T}, \,\, \text{for all} \,\, z \in \mathbb{D}.
		\end{aligned}
	\end{equation}
	Note that $\Theta_T \in \mathcal{B}(\mathcal{D}_T, \mathcal{D}_{T^*})$. We define a multiplication operator $M_{\Theta_T} : H^2(\mathbb{D}) \otimes \mathcal{D}_T \to H^2(\mathbb{D}) \otimes \mathcal{D}_{T^*}$ by
	\begin{equation}\label{M_Theta_T}
		\begin{aligned}
			M_{\Theta_T}f(z) = \Theta_T(z)f(z) \,\, \text{for} \,\, z \in \mathbb{D},
		\end{aligned}
	\end{equation}
	and also define $\mathcal{H}_T = (H^2(\mathbb{D}) \otimes \mathcal{D}_{T^*}) \ominus M_{\Theta_T}(H^2(\mathbb{D}) \otimes \mathcal{D}_T)$. We call $\mathcal{H}_T$ the \textit{model space for $T$}. The following theorem describes the functional model for pure contraction \cite{Nagy}.
	
	\begin{thm}\label{Pure Contraction Model}
		Every pure contraction $T$ defined on a Hilbert space $\mathcal{H}$ is unitarily 
		equivalent to the operator $T_1$ on the Hilbert space $\mathcal{H}_T = (H^2(\mathbb{D}) \otimes \mathcal{D}_{T^*}) \ominus M_{\Theta_T}(H^2(\mathbb{D}) \otimes \mathcal{D}_T)$ defined as
		\begin{equation}\label{Model}
			\begin{aligned}
				T_1 &= P_{\mathcal{H}_T}
				(M_z \otimes I_{\mathcal{D}_{T^*}})|_{\mathcal{H}_T}.
			\end{aligned}
		\end{equation}
	\end{thm}
	
	Let $\mathcal M_{n\times n}(\mathbb{C})$ be the set of all $n\times n$ complex matrices and  $E$ be a linear subspace of $\mathcal M_{n\times n}(\mathbb{C}).$ We define the function $\mu_{E}: \mathcal M_{n\times n}(\mathbb{C}) \to [0,\infty)$ as follows:
\begin{equation}\label{mu}
\mu_{E}(A):=\frac{1}{\inf\{\|X\|: \,\ \det(1-AX)=0,\,\, X\in E\}},\;\; A\in \mathcal M_{n\times n}(\mathbb{C})
	\end{equation}
with the understanding that $\mu_{E}(A):=0$ if $1-AX$ is  nonsingular for all $X\in E$ \cite{ds, jcd}.   Here $\|\cdot\|$ denotes the operator norm. Let  $E(n;s;r_{1},\dots,r_{s})\subset \mathcal M_{n\times n}(\mathbb{C})$ be the vector subspace comprising block diagonal matrices, defined as follows:
\begin{equation}\label{ls}
    	E=E(n;s;r_{1},...,r_{s}):=\{\operatorname{diag}[z_{1}I_{r_{1}},....,z_{s}I_{r_{s}}]\in \mathcal M_{n\times n}(\mathbb{C}): z_{1},...,z_{s}\in \mathbb{C}\},
\end{equation}
 where $\sum_{i=1}^{s}r_i=n.$ We recall the definition of $\Gamma_{E(3; 3; 1, 1, 1)}$, $\Gamma_{E(3; 2; 1, 2)}$ and $\Gamma_{E(2; 2; 1, 1)}$ \cite{Abouhajar,Bharali, apal1}. The sets $\Gamma_{E{(2;2;1,1)}}$, $\Gamma_{E(3; 3; 1, 1, 1)}$ and $\Gamma_{E(3; 2; 1, 2)}$ are defined as 
 \begin{equation*}
\begin{aligned}
\Gamma_{E{(2;2;1,1)}}:=\Big \{\textbf{x}=(x_1=a_{11}, x_2=a_{22}, x_3=a_{11}a_{22}-a_{12}a_{21}=\det A)\in \mathbb C^3: &\\A\in \mathcal M_{2\times 2}(\mathbb C)~{\rm{and}}~\mu_{E(2;2;1,1)}(A)\leq 1\Big \},
\end{aligned}
\end{equation*}	 
 \begin{equation*}
\begin{aligned}
\Gamma_{E{(3;3;1,1,1)}}:=\Big \{\textbf{x}=(x_1=a_{11}, x_2=a_{22}, x_3=a_{11}a_{22}-a_{12}a_{21}, x_4=a_{33}, x_5=a_{11}a_{33}-a_{13}a_{31},&\\ x_6=a_{22}a_{33}-a_{23}a_{32},x_7=\det A)\in \mathbb C^7: A\in \mathcal M_{3\times 3}(\mathbb C)~{\rm{and}}~\mu_{E(3;3;1,1,1)}(A)\leq 1\Big \}
\end{aligned}
\end{equation*}	
$${\rm{and}}$$  
\begin{equation*}
\begin{aligned}
\Gamma_{E(3;2;1,2)}:=\Big\{( x_1=a_{11},x_2=\det \left(\begin{smallmatrix} a_{11} & a_{12}\\
					a_{21} & a_{22}
				\end{smallmatrix}\right)+\det \left(\begin{smallmatrix}
					a_{11} & a_{13}\\
					a_{31} & a_{33}
				\end{smallmatrix}\right),x_3=\operatorname{det}A, y_1=a_{22}+a_{33}, &\\ y_2=\det  \left(\begin{smallmatrix}
					a_{22} & a_{23}\\
					a_{32} & a_{33}\end{smallmatrix})\right)\in \mathbb C^5
:A\in \mathcal M_{3\times 3}(\mathbb C)~{\rm{and}}~\mu_{E(3;2;1,2)}(A)\leq 1\Big\}.
\end{aligned}
\end{equation*}

The sets $\Gamma_{E(3; 2; 1, 2)}$ and $\Gamma_{E(2; 2; 1, 1)}$  are referred to as $\mu_{1,3}-$\textit{quotient} and tetrablock, respectively \cite{Abouhajar, Bharali}.  Studying the symmetrized bidisc and the tetrablock is essential in complex analysis and operator theory. Young's investigation of the symmetrized bidisc and the tetrablock, in collaboration with several co-authors \cite{Abouhajar, JAgler, young, ay, ay1, JAY, JANY}, has also been carried out through an operator-theoretic point of view. Agler and Young established normal dilation for a pair of commuting operators with the symmetrized bidisc as a spectral set \cite{JAgler, young}. Various authors have investigated the properties of $\Gamma_n$-isometries, $\Gamma_n$ unitaries, the Wold decomposition, and conditional dilation of $\Gamma_n$ \cite{SS, A. Pal}. T. Bhattacharyya studied the tetrablock isometries, tetrablock unitaries, the Wold decomposition for tetrablock, and conditional dilation for tetrablock \cite{Bhattacharyya}. However, whether the tetrablock and $\Gamma_n, n>3,$ have the property $P$ remains unresolved.
	
	Let  $$K=\{\textbf{x}=(x_1,\ldots,x_7)\in \Gamma_{E(3;3;1,1,1)} :x_1=\bar{x}_6x_7, x_3=\bar{x}_4x_7,x_5=\bar{x}_2x_7 ~{\rm{and}}~|x_7|=1\}$$ 
$$\rm{and}$$
\[ K_1 = \{x = (x_1, x_2, x_3, y_1, y_2) \in \Gamma_{E(3;2;1,2)} : x_1 = \overline{y}_2 x_3, x_2 = \overline{y}_1 x_3, |x_3| = 1 \}. \]

We begin with the following definitions that will be essential for our discussion.

\begin{defn}\label{def-1}
		\begin{enumerate}
			\item If $\Gamma_{E(3; 3; 1, 1, 1)}$ is a spectral set for $\textbf{T} = (T_1, \dots, T_7)$, then the $7$-tuple of commuting bounded operators $\textbf{T}$ defined on a  Hilbert space $\mathcal{H}$ is referred to as a \textit{$\Gamma_{E(3; 3; 1, 1, 1)}$-contraction}.
			
			\item Let $(S_1, S_2, S_3)$ and $(\tilde{S}_1, \tilde{S}_2)$ be tuples of commuting bounded operators defined on a Hilbert space $\mathcal{H}$ with $S_i\tilde{S}_j = \tilde{S}_jS_i$ for $1 \leqslant i \leqslant 3$ and $1 \leqslant j \leqslant 2$. We say that  $\textbf{S} = (S_1, S_2, S_3, \tilde{S}_1, \tilde{S}_2)$ is a $\Gamma_{E(3; 2; 1, 2)}$-contraction if $ \Gamma_{E(3; 2; 1, 2)}$ is a spectral set for $\textbf{S}$.
			
\item A commuting $7$-tuple of normal operators $\textbf{N} = (N_1, \dots, N_7)$ defined on a Hilbert space $\mathcal{H}$ is  a \textit{$\Gamma_{E(3; 3; 1, 1, 1)}$-unitary} if the Taylor joint spectrum $\sigma(\textbf{N})$ is contained in the set $K$.  		
			\item A commuting $5$-tuple of normal operators $\textbf{M} = (M_1, M_2, M_3, \tilde{M}_1, \tilde{M}_2)$ on a Hilbert space $\mathcal{H}$ is referred as a \textit{$\Gamma_{E(3; 2; 1, 2)}$-unitary} if the Taylor joint spectrum $\sigma(\textbf{M})$ is contained in $K_1.$ 
					
\item A  $\Gamma_{E(3; 3; 1, 1, 1)}$-isometry (respectively, $\Gamma_{E(3; 2; 1, 2)}$-isometry) is defined as the restriction of a $\Gamma_{E(3; 3; 1, 1, 1)}$-unitary (respectively, $\Gamma_{E(3; 2; 1, 2)}$-unitary)  to a joint invariant subspace. In other words, a $\Gamma_{E(3; 3; 1, 1, 1)}$-isometry ( respectively, $\Gamma_{E(3; 2; 1, 2)}$-isometry) is a $7$-tuple (respectively, $5$-tuple) of commuting bounded operators that possesses simultaneous extension to a \textit{$\Gamma_{E(3; 3; 1, 1, 1)}$-unitary} (respectively, \textit{$\Gamma_{E(3; 2; 1, 2)}$-unitary}). It is important to observe that a $\Gamma_{E(3; 3; 1, 1, 1)}$-isometry (respectively, $\Gamma_{E(3; 2; 1, 2)}$-isometry ) $\textbf{V}=(V_1\dots,V_7)$ (respectively,  $\textbf{W}=(W_1,W_2,W_3,\tilde{W}_1,\tilde{W}_2)$) consists of commuting subnormal operators with $V_7$ (respectively, $W_3$)  is an isometry.

\item   We say that $\textbf{V}$ (respectively, $\textbf{W}$)   is a pure $\Gamma_{E(3; 3; 1, 1, 1)}$-isometry (respectively, pure $\Gamma_{E(3; 2; 1, 2)}$-isometry) if $V_7$ (respectively, $W_3$) is a  pure isometry,  that is, a shift of some multiplicity.		
\end{enumerate}
	\end{defn}	
\begin{defn}
		\begin{enumerate}
			\item A $\Gamma_{E(3; 3; 1, 1, 1)}$-contraction $\textbf{T} = (T_1, \dots, T_7)$ is said to be completely non-unitary  {$\Gamma_{E(3; 3; 1, 1, 1)}$-contraction} if $T_7$ is a completely non-unitary contraction.
			
			\item A $\Gamma_{E(3; 2; 1, 2)}$-contraction $\textbf{S} = (S_1, S_2, S_3, \tilde{S}_1, \tilde{S}_2)$ is said to be completely non-unitary {$\Gamma_{E(3; 2; 1, 2)}$-contraction} if $S_3$ is a completely non-unitary contraction.		\end{enumerate}
	\end{defn}

We denote the unit circle by $\mathbb T.$  Let $\mathcal E$  be a separable Hilbert space. Let  $\mathcal B(\mathcal E)$ denote the space of bounded linear operators on $\mathcal E$ equipped with the operator norm. Let $H^2(\mathcal E)$ denote  the  Hardy space of analytic $\mathcal E$-valued functions defined on the unit disk  $\mathbb D$. Let $ L^2(\mathcal E)$ represent the Hilbert space of square-integrable $\mathcal E$-valued functions on the unit circle $\mathbb T,$ equipped with the natural inner product. The space $H^{\infty}(\mathcal B(\mathcal E))$ consists of bounded analytic $\mathcal B(\mathcal E)$-valued functions defined on $\mathbb D$. Let $L^{\infty}(\mathcal B(\mathcal E))$ denote the space of bounded measurable $\mathcal B(\mathcal E)$-valued functions on $\mathbb T$.  For $\phi \in L^{\infty}(\mathcal B(\mathcal E)),$ the Toeplitz operator associated with the symbol  $\phi$ is denoted by $T_{\phi}$ and is defined as follows: 
$$T_{\phi}f=P_{+}(\phi f), f \in H^2(\mathcal E),$$ where $P_{+} : L^2(\mathcal E) \to H^2(\mathcal E)$ is the orthogonal projecton.  In particular, $T_z$ is the
unilateral shift operator $M_z$ on $H^2(\mathcal E)$  and $T_{\bar{z}}$ is the backward shift $M_z^*$ on $H^2(\mathcal E)$.

We recall fundamental equations and fundamental operator for $\Gamma_{E(3; 3; 1, 1, 1)}$-contraction (respectively, $\Gamma_{E(3; 2; 1, 2)}$-contraction) from \cite{apal2}.
	
	\begin{defn}\label{fundamental}
		Let $(T_1, \dots, T_7)$ be a $7$-tuple of commuting contractions on a Hilbert space $\mathcal{H}. $ The equations 

		\begin{equation}\label{Fundamental 1}
\begin{aligned}
&T_i - T^*_{7-i} T_7 = D_{T_7}F_iD_{T_7}, \;\;\; 1\leq i\leq 6, 
\end{aligned}
\end{equation}
where $F_i\in \mathcal{B}(\mathcal{D}_{T_7}),$ are referred to as the  fundamental equations for $(T_1, \dots, T_7)$.
			\end{defn}
	
\begin{defn}\label{fundamental}
Let $(S_1, S_2, S_3, \tilde{S}_1, \tilde{S}_2)$ be a $5$-tuple of commuting bounded operators defined on a Hilbert space $\mathcal H$. The equations
\begin{equation}
			\begin{aligned}\label{funda1}
					&S_1 - \tilde{S}^*_2S_3 = D_{S_3}G_1D_{S_3},\,\, \tilde{S}_2 - S^*_1S_3 = D_{S_3}\tilde{G}_2D_{S_3},
				\end{aligned}
			\end{equation}
			$${\rm{and}}$$
		\begin{equation}
				\begin{aligned}\label{funda11}
				&\frac{S_2}{2} - \frac{\tilde{S}^*_1}{2}S_3 = D_{S_3}G_2D_{S_3}, \,\, \frac{\tilde{S}_1}{2} - \frac{S^*_2}{2}S_3 = D_{S_3}\tilde{G}_1D_{S_3},				\end{aligned}
			\end{equation}
where $G_1,2G_2,2\tilde{G}_1$ and $\tilde{G}_2$ in $\mathcal{B}(\mathcal{D}_{S_3}),$ are referred to as the  fundamental equations for $(S_1, S_2, S_3, \tilde{S}_1, \tilde{S}_2)$.		\end{defn}

	In section \ref{Admissible}, we show that for a given pure contraction $T_7$ acting on a Hilbert space $\mathcal{H}$, if $(\tilde{F}_1, \dots, \tilde{F}_6) \in \mathcal{B}(\mathcal{D}_{T^*_7})$ with $[\tilde{F}_i, \tilde{F}_j] = 0, [\tilde{F}^*_i, \tilde{F}_{7-j}] = [\tilde{F}^*_j, \tilde{F}_{7-i}]$,$w(\tilde{F}^*_i + \tilde{F}_{7-i}z) \leqslant 1$ and these operators satisfy
		\[(\tilde{F}^*_i + \tilde{F}_{7-i}z)\Theta_{T_7}(z) = \Theta_{T_7}(z)(F_i + F^*_{7-i}z) \,\, \text{for all} \,\, z \in \mathbb{D}\] for $1 \leqslant i, j \leqslant 6$ for some $(F_1, \dots, F_6) \in \mathcal{B}(\mathcal{D}_{T_7})$ with $w(F^*_i + F_{7-i}z) \leqslant 1$ for $1 \leqslant i \leqslant 6$, then there exists a $\Gamma_{E(3; 3; 1, 1, 1)}$-contraction $(T_1, \dots, T_7)$ such that $F_1, \dots, F_6$ are the fundamental operators of $(T_1, \dots, T_7)$ and $\tilde{F}_1, \dots, \tilde{F}_6$ are the fundamental operators of $(T^*_1, \dots, T^*_7)$. We also establish analogous results for pure $\Gamma_{E(3; 2; 1, 2)}$-contractions.
In Section \ref{Canonical Construction}, we present the explicit construction of a $\Gamma_{E(3; 3; 1, 1, 1)}$-unitary (respectively, $\Gamma_{E(3; 2; 1, 2)}$-unitary) arising from a $\Gamma_{E(3; 3; 1, 1, 1)}$-contraction (respectively, $\Gamma_{E(3; 2; 1, 2)}$-contraction). Section \ref{Model for Isometries} develops functional models for $\Gamma_{E(3; 3; 1, 1, 1)}$- and $\Gamma_{E(3; 2; 1, 2)}$-isometries. Finally, in Section \ref{Douglas Type Functional Model}, we study Douglas-type functional models for $\Gamma_{E(3; 3; 1, 1, 1)}$- and $\Gamma_{E(3; 2; 1, 2)}$-contractions. In Section \ref{Nagy-Foias Type Functional Model}, we develop Nagy-Foias type functional models for completely non-unitary (c.n.u.) $\Gamma_{E(3; 3; 1, 1, 1)}$-contractions and, analogously, for c.n.u. $\Gamma_{E(3; 2; 1, 2)}$-contractions. Section \ref{Schaffer Type Model} is dedicated to constructing the \textit{Schaffer type model} for the corresponding $\Gamma_{E(3; 3; 1, 1, 1)}$-isometric and $\Gamma_{E(3; 2; 1, 2)}$-isometric dilations.

	\section{Admissible Fundamental Operators of $\Gamma_{E(3; 3; 1, 1, 1)}$-Contraction and $\Gamma_{E(3; 2; 1, 2)}$-Contraction}\label{Admissible}
	
	In this section, we examine that for a given contraction $T_7$ and $(F_1, \dots, F_6) \in \mathcal{B}(\mathcal{D}_{T_7})$ and $(\tilde{F}_1, \dots, \tilde{F}_6) \in \mathcal{B}(\mathcal{D}_{T^*_7})$ such that $w(\tilde{F}^*_i + \tilde{F}_{7-i}z) \leqslant 1$ and $w(F^*_i + F_{7-i}z) \leqslant 1$ for $1 \leqslant i \leqslant 6$, does there always exist a $\Gamma_{E(3; 3; 1, 1, 1)}$-contraction $\textbf{T} = (T_1, \dots, T_7)$ such that $F_1, \dots, F_6$ are the fundamental operators of $\textbf{T}$ and $\tilde{F}_1, \dots, \tilde{F}_6$ the fundamental operators of $\textbf{T}^*$? Similarly, we investigate that when $S_3$ is a given contraction on a Hilbert space $\mathcal{H}$ and $(G_1, 2G_2, 2\tilde{G}_1, \tilde{G}_2) \in \mathcal{B}(\mathcal{D}_{S_3})$ and $(\hat{G}_1, 2\hat{G}_2, 2\hat{\tilde{G}}_1, \hat{\tilde{G}}_2) \in \mathcal{B}(\mathcal{D}_{S^*_3})$ such that $w(G^*_1 + \tilde{G}_2z) \leqslant 1, w(G^*_2 + \tilde{G}_1z) \leqslant 1$ and $w(\hat{G}^*_1 + \hat{\tilde{G}}_2z) \leqslant 1, w(\hat{G}^*_2 + \hat{\tilde{G}}_1z) \leqslant 1$ then does there exists any $\Gamma_{E(3; 2; 1, 2)}$-contraction $\textbf{S} = (S_1, S_2, S_3, \tilde{S}_1, \tilde{S}_2)$ such that $G_1, 2G_2, 2\tilde{G}_1, \tilde{G}_2$ are the fundamental operators of $\textbf{S}$ and $\hat{G}_1, 2\hat{G}_2, 2\hat{\tilde{G}}_1, \hat{\tilde{G}}_2$ are the fundamental operators of $\textbf{S}^*$?
	
	Let $\textbf{T} = (T_1, \dots, T_7)$ be a $\Gamma_{E(3; 3; 1, 1, 1)}$-contraction. Thus $T_7$ is a contraction. We define $W : \mathcal{H} \to H^2(\mathbb{D}) \otimes \mathcal{D}_{T^*_7}$ by
	\begin{equation}\label{W}
		\begin{aligned}
			W(h) &= \sum_{n \geqslant 0} z^n \otimes D_{T^*_7}T^{*n}_7h.
		\end{aligned}
	\end{equation}
	Since $T_7$ is a pure isometry, one can easily deduced that $W$ is isometry. The adjoint of $W$ is given by 
	\begin{equation}\label{W*}
		\begin{aligned}
			W^*(z^n \otimes \xi) &= T^n_7D_{T^*_7}\xi \,\, \text{for} \,\, n \in \mathbb{N} \cup \{0\}, \xi \in \mathcal{D}_{T^*_7}.
		\end{aligned}
	\end{equation}
	
	In the following we state an well known result for a contraction. We write it in terms of our terminologies.
	
	\begin{prop}[Proposition $2.2$, \cite{apal3}]\label{FiFj}
		The fundamental operators $F_i$ and $F_{7-i}$ for $1\leq i \leq 6$ of $\Gamma_{E(3; 3; 1, 1, 1)}$-contraction $\textbf{T} = (T_1, \dots, T_7)$ are the unique bounded linear operators $X_i$ and $X_{7-i}$ for $1\leq i \leq 6$ on $\mathcal D_{T_7}$ that satisfy the following operator equations: 
		\begin{equation}\label{FP 1}
			\begin{aligned}
				&D_{T_7}T_i = X_iD_{T_7} + X^*_{7-i}D_{T_7}T_7 ~\text{and}~ D_{T_7}T_{7-i} = X_{7-i}D_{T_7} + X^*_iD_{T_7}T_7~{\rm{for}}~1\leq i \leq 6.
			\end{aligned}
		\end{equation}
	\end{prop}
	
	We recall some results on $\Gamma_{E(3; 3; 1, 1, 1)}$-contraction that are crucial for the discussion of the main result of this section.
	
	\begin{lem}[Lemma $3.1$, \cite{apal3}]\label{Lem 1}
		Let $T_7$ be a contraction. Then
		\begin{equation}\label{W Property}
			\begin{aligned}
				WW^* + M_{\Theta_{T_7}}M^*_{\Theta_{T_7}} = I_{H^2(\mathbb{D}) \otimes \mathcal{D}_{T^*_7}}
			\end{aligned}
		\end{equation}
		holds.
	\end{lem}
	
	We only state the following result. For proof see [Theorem $2.7$, \cite{apal3}].
	
	\begin{thm}[Theorem $2.7$, \cite{apal3}]\label{Thm 1}
		Let  $F_i, 1\leq i \leq 6$ be fundamental operators of a $\Gamma_{E(3; 3; 1, 1, 1)}$-contraction $\textbf{T} = (T_1, \dots, T_7)$  and $\tilde{F}_j,1\leq j \leq 6$ be  fundamental operators of a $\Gamma_{E(3; 3; 1, 1, 1)}$-contraction $\textbf{T}^* = (T^*_1, \dots, T^*_7)$. Then
		\begin{equation}\label{Fundamental P4}
			\begin{aligned}
				(F^*_i + F_{7-i}z)\Theta_{T^*_7}(z) &= \Theta_{T^*_7}(z)(\tilde{F}_i + \tilde{F}^*_{7-i}z) \,\, \textit{for} \,\, 1 \leqslant i \leqslant 6~{\rm{and~for ~all }}~ z \in \mathbb{D}.
			\end{aligned}
		\end{equation}
	\end{thm}
	
	The following theorem is one of the main results of this article.
	
	\begin{thm}\label{Thm 2}
		Let $F_1, \dots, F_6$ be the fundamental operators of the $\Gamma_{E(3; 3; 1, 1, 1)}$-contraction $\textbf{T} = (T_1, \dots, T_7)$ and $\tilde{F}_1, \dots, \tilde{F}_6$ be the fundamental operators of the $\Gamma_{E(3; 3; 1, 1, 1)}$-contraction $\textbf{T}^* = (T^*_1, \dots, T^*_7)$. Then
		\begin{equation}\label{Admissible 1}
			\begin{aligned}
				(\tilde{F}^*_i + \tilde{F}_{7-i}z)\Theta_{T_7}(z) &= \Theta_{T_7}(z)(F_i + F^*_{7-i}z) \,\, \text{for all $z \in \mathbb{D} \,\, \text{and} \,\, 1 \leqslant i \leqslant 6$}.
			\end{aligned}
		\end{equation}
		
		Conversely, let $T_7$ be a pure contraction on a Hilbert space $\mathcal{H}$. Let $\tilde{F}_1, \dots, \tilde{F}_6 \in \mathcal{B}(\mathcal{D}_{T^*_7})$ such that $w(\tilde{F}^*_i + \tilde{F}_{7-i}z) \leqslant 1$ for $1 \leqslant i \leqslant 6$ that satisfy
		\begin{equation}\label{Commutative 1}
			\begin{aligned}
				&[\tilde{F}_i, \tilde{F}_j] = 0 \,\, \text{and} \,\, [\tilde{F}^*_i, \tilde{F}_{7-j}] = [\tilde{F}^*_j, \tilde{F}_{7-i}] \,\, \text{for} \,\, 1 \leqslant i, j \leqslant 6.
			\end{aligned}
		\end{equation}
		If $\tilde{F}_1, \dots, \tilde{F}_6$ satisfy \eqref{Admissible 1} for some $F_1, \dots, F_6 \in \mathcal{B}(\mathcal{D}_{T_7})$ such that $w(F^*_i + F_{7-i}z) \leqslant 1$ for $1 \leqslant i \leqslant 6$ then there exists a $\Gamma_{E(3; 3; 1, 1, 1)}$-contraction $\textbf{T} = (T_1, \dots, T_7)$ such that $F_1, \dots, F_6$ be the fundamental operators of $\textbf{T}$ and $\tilde{F}_1, \dots, \tilde{F}_6$ be the fundamental operators of $\textbf{T}^* = (T^*_1, \dots, T^*_7)$.
	\end{thm}
	
	\begin{proof}
		One direction can be deduced by applying Theorem \ref{Thm 1} for the $\Gamma_{E(3; 3; 1, 1, 1)}$-contraction $\textbf{T}^*$.
		
		In order to prove the other direction, let $W$ be the isometry defined in \eqref{W}. Notice that
		\begin{equation}\label{Ad 1.1}
			\begin{aligned}
				M^*_zWh
				&= M^*_z\left(\sum_{n \geqslant 0} z^nD_{T^*_7}T^{*n}_7h\right)
				= \sum_{n \geqslant 0} z^nD_{T^*_7}T^{*n+1}_7h = WT^*_7h.
			\end{aligned}
		\end{equation}
		Thus $M^*_zW = WT^*_7$. Now we define
		\begin{equation}\label{T_i}
			\begin{aligned}
				T_i = W^*M_{\tilde{F}^*_i + \tilde{F}_{7-i}z}W \,\,\text{for}\,\, 1 \leqslant i \leqslant 6.
			\end{aligned}
		\end{equation}
		As $T_7$ is a pure contraction, we have $(\Ran W)^{\perp} = \Ran M_{\Theta_{T_7}}$. By $\eqref{Admissible 1}$ we have that $\Ran M_{\Theta_{T_7}}$ is invariant under $M_{\tilde{F}^*_i + \tilde{F}_{7-i}z}$; i.e., $\Ran W$ co-invariant under $M_{\tilde{F}^*_i + \tilde{F}_{7-i}z}$. We show that $\textbf{T} = (T_1, \dots, T_7)$ is a $\Gamma_{E(3; 3; 1, 1, 1)}$-contraction. First we show that $T_1, \dots, T_7$ commute. Note that
		\begin{equation}\label{Ad 1.2}
			\begin{aligned}
				T^*_iT^*_j
				&= W^*M^*_{\tilde{F}^*_i + \tilde{F}_{7-i}z}WW^*M^*_{\tilde{F}^*_j + \tilde{F}_{7-j}z}W\\
				&= W^*M^*_{\tilde{F}^*_i + \tilde{F}_{7-i}z}M^*_{\tilde{F}^*_j + \tilde{F}_{7-j}z}W \,\, (\text{as $WW^*$ is a projection onto $\Ran W$})\\
				&= W^*M^*_{\tilde{F}^*_j + \tilde{F}_{7-j}z}M^*_{\tilde{F}^*_i + \tilde{F}_{7-i}z}W \,\, (\text{by \eqref{Commutative 1}})\\
				&= W^*M^*_{\tilde{F}^*_j + \tilde{F}_{7-j}z}WW^*M^*_{\tilde{F}^*_i + \tilde{F}_{7-i}z}W \,\, (\text{as $WW^*$ is a projection onto $\Ran W$})\\
				&= T^*_jT^*_i.
			\end{aligned}
		\end{equation}
		Thus we have $T_iT_j = T_jT_i$ for $1 \leqslant i, j \leqslant 6$. To show that $T_i$ commutes with $T_7$ we proceed as follows:
		\begin{equation}\label{Ad 1.3}
			\begin{aligned}
				T^*_iT^*_7
				&= W^*M^*_{\tilde{F}^*_i + \tilde{F}_{7-i}z}WW^*M^*_zW\\
				&= W^*M^*_{\tilde{F}^*_i + \tilde{F}_{7-i}z}M^*_zW \,\, (\text{as $WW^*$ is a projection onto $\Ran W$})\\
				&= W^*M^*_zM^*_{\tilde{F}^*_i + \tilde{F}_{7-i}z}W \,\, (\text{since $M_z$ commutes with $M_{\tilde{F}^*_j + \tilde{F}_{7-j}z}$})\\
				&= W^*M^*_zWW^*M^*_{\tilde{F}^*_i + \tilde{F}_{7-i}z}W  \,\, (\text{as $WW^*$ is a projection onto $\Ran W$})\\
				&= T^*_7T^*_i.
			\end{aligned}
		\end{equation}
		Hence $T_i$ commutes with $T_7$ for $1 \leqslant i \leqslant 6$. By \eqref{Commutative 1} that $(M_{\tilde{F}^*_1 + \tilde{F}_6z}, \dots, M_{\tilde{F}^*_6 + \tilde{F}_1z}, M_z)$ a commuting $7$-tuple of operators. Then by [Theorem $4.6$, \cite{apal1}] it is clear that $(M_{\tilde{F}^*_1 + \tilde{F}_6z}, \dots, M_{\tilde{F}^*_6 + \tilde{F}_1z}, M_z)$ is a $\Gamma_{E(3; 3; 1, 1, 1)}$-isometry. We prove that $\textbf{T}$ is a $\Gamma_{E(3; 3; 1, 1, 1)}$-contraction. It follows from \eqref{Ad 1.1} and \eqref{T_i} that for any polynomial $p$ in $7$ variables
		\begin{equation}\label{Ad 1.4}
			\begin{aligned}
				p(T^*_1, \dots, T^*_6, T^*_7)
				&= W^*p(M^*_{\tilde{F}^*_1 + \tilde{F}_6z}, \dots, M^*_{\tilde{F}^*_6 + \tilde{F}_1z}, M^*_z)W.
			\end{aligned}
		\end{equation}
		It yields from \eqref{Ad 1.4} that
		\begin{equation*}
			\begin{aligned}
				||p(T^*_1, \dots, T^*_6, T^*_7)||
				&= ||W^*p(M^*_{\tilde{F}^*_1 + \tilde{F}_6z}, \dots, M^*_{\tilde{F}^*_6 + \tilde{F}_1z}, M^*_z)W||\\
				&\leqslant
				||p(M^*_{\tilde{F}^*_1 + \tilde{F}_6z}, \dots, M^*_{\tilde{F}^*_6 + \tilde{F}_1z}, M^*_z)||\\
				&\leqslant
				||p||_{\infty, \Gamma_{E(3; 3; 1, 1, 1)}}.
			\end{aligned}
		\end{equation*}
		This implies that $\textbf{T}^*$ is a $\Gamma_{E(3; 3; 1, 1, 1)}$-contraction and hence $\textbf{T}$ is a $\Gamma_{E(3; 3; 1, 1, 1)}$-contraction.
		
		Next we show that $\tilde{F}_1, \dots, \tilde{F}_6$ are the fundamental operators of $\textbf{T}^*$. Notice that for $1 \leqslant i \leqslant 6$
		\begin{equation}\label{Ad 1.5}
			\begin{aligned}
				T^*_i - T_{7-i}T^*_7
				&= W^*M^*_{\tilde{F}^*_i + \tilde{F}_{7-i}z}W - W^*M_{\tilde{F}^*_{7-i} + \tilde{F}_iz}WW^*M^*_zW\\
				&= W^*M^*_{\tilde{F}^*_i + \tilde{F}_{7-i}z}W - W^*M_{\tilde{F}^*_{7-i} + \tilde{F}_iz}M^*_zW \,\, (\text{as $WW^*$ is a projection onto $\Ran W$})\\
				&= W^*((I \otimes \tilde{F}_i) + (M^*_z \otimes \tilde{F}^*_{7-i}) - (M^*_z \otimes \tilde{F}^*_{7-i}) - (M_zM^*_z \otimes \tilde{F}_i))W\\
				&= W^*((I - M_zM^*_z) \otimes \tilde{F}_i)W\\
				&= W^*(\mathbb{P}_{\mathbb{C}} \otimes \tilde{F}_i)W\\
				&= D_{T^*_7}\tilde{F}_iD_{T^*_7}.
			\end{aligned}
		\end{equation}
		Thus it follows that $T^*_i - T_{7-i}T^*_7 = D_{T^*_7}\tilde{F}_iD_{T^*_7}$ and hence $\tilde{F}_1, \dots, \tilde{F}_6$ are the fundamental operators of $\textbf{T}^*$. Suppose $X_1, \dots, X_6$ are the fundamental operators of $\textbf{T}$. Then by part one of this theorem we get
		\begin{equation}\label{Ad 1.6}
			\begin{aligned}
				(\tilde{F}^*_i + \tilde{F}_{7-i}z)\Theta_{T_7}(z) &= \Theta_{T_7}(z)(X_i + X^*_{7-i}z) \,\, \text{for all $z \in \mathbb{D}, 1 \leqslant i \leqslant 6$}.
			\end{aligned}
		\end{equation}
		By \eqref{Ad 1.6} and \eqref{Admissible 1} we have that
		\begin{equation}\label{Ad 1.7}
			\begin{aligned}
				\Theta_{T_7}(z)(X_i + X^*_{7-i}z) &= \Theta_{T_7}(z)(F_i + F^*_{7-i}z) \,\, \text{for $1 \leqslant i \leqslant 6$}.
			\end{aligned}
		\end{equation}
		As $T_7$ is pure, $M_{\Theta_{T_7}}$ is an isometry and hence from \eqref{Ad 1.7} we obtain $X_i + X^*_{7-i}z = F_i + F^*_{7-i}z$ for $1 \leqslant i \leqslant 6$ and $ z \in \mathbb{D}$. From here it is immediate that $X_i = F_i$ for $1 \leqslant i \leqslant 6$. Therefore, $F_1, \dots, F_6$ are the fundamental operators of $\textbf{T}$. This completes the proof.
	\end{proof}
	
	In the subsequent part of this section, we develop an analogous result for $\Gamma_{E(3; 2; 1, 2)}$-contraction. Let $\textbf{S} = (S_1, S_2, S_3, \tilde{S}_1, \tilde{S}_2)$ be a $\Gamma_{E(3; 2; 1, 2)}$-contraction. Then $S_3$ is a contraction. We define
	$\tilde{W} : \mathcal{H} \to H^2(\mathbb{D}) \otimes \mathcal{D}_{S^*_3}$ by \begin{equation}\label{W tilde}
		\begin{aligned}
			\tilde{W}(h) &= \sum_{n \geqslant 0} z^n \otimes D_{S^*_3}S^{*n}_3h
		\end{aligned}
	\end{equation}
	As $S_3$ is an isometry, we deduce that $\tilde{W}$ is an isometry. The adjoint of $\tilde{W}^*$ has the following form
	\begin{equation}\label{W tilde*}
		\begin{aligned}
			\tilde{W}^*(z^n \otimes \eta) &= S^n_3D_{S^*_3}\eta \,\, \text{for} \,\, n \in \mathbb{N} \cup \{0\}, \eta \in \mathcal{D}_{S^*_3}.
		\end{aligned}
	\end{equation}
	
	\begin{prop}[Proposition $2.9$, \cite{apal3}]\label{s1s3}
		The fundamental operators  of a $\Gamma_{E(3; 2; 1, 2)}$-contraction $\textbf{S} = (S_1, S_2, S_3,\\ \tilde{S}_1, \tilde{S}_2)$ are the unique operators $G_1,\tilde{G}_2,G_2$ and $\tilde{G}_1$  defined on $\mathcal{D}_{S_3}$ which satisfy the following operator equations
		\begin{equation}\label{s3}
			\begin{aligned}
				&D_{S_3}S_1 = G_1D_{S_3} + \tilde{G}_2^*D_{S_3}S_3, \,\, D_{S_3}\tilde{S}_2 = \tilde{G}_2D_{S_3} + G_1^*D_{S_3}S_3, \\&\,\, ~~~~~~~~~~~~~~~~~~~~~~~~~~~~~~~~~~~~~~~~~~~~\,\,\,\,\,\,\,\,\,\,\,\,\,\,\,\,\,\,\,\,\,\,\,\,\,\,\,\,\,\,\,\,\,\,\,\,\,\,\,\,\,\,\,\,\,\,\,\,\,\,\,\,\ \text{and}\\
				&D_{S_3}\frac{S_2}{2} = G_2D_{S_3} + \tilde{G}^*_1D_{S_3}S_3, \,\, D_{S_3}\frac{\tilde{S}_1}{2} = \tilde{G}_1D_{S_3} + G^*_2D_{S_3}S_3.
			\end{aligned}
		\end{equation}
	\end{prop}
	
	We recall some results on $\Gamma_{E(3; 2; 1, 2)}$-contraction that play important roll in the remain part of this section.
	
	\begin{lem}[Lemma $3.1$, \cite{apal3}]\label{Lem 2}
		If $S_3$ is a contraction then
		\begin{equation}\label{W tilde Property}
			\begin{aligned}
				\tilde{W}\tilde{W}^* + M_{\Theta_{S_3}}M^*_{\Theta_{S_3}} = I_{H^2(\mathbb{D}) \otimes \mathcal{D}_{S^*_3}}.
			\end{aligned}
		\end{equation}
	\end{lem}
	
	We only state the following result. For proof see [Theorem $2.7$, \cite{apal3}].
	
	\begin{thm}[Theorem $2.14$, \cite{apal3}]\label{Thm 3}
		Let $\textbf{S} = (S_1, S_2, S_3, \tilde{S}_1, \tilde{S}_2)$ be a $\Gamma_{E(3; 2; 1, 2)}$-contraction on a Hilbert space $\mathcal{H}$. Suppose $G_1, 2G_2, 2\tilde{G}_1, \tilde{G}_2$ and $\hat{G}_1, 2\hat{G}_2, 2\hat{\tilde{G}}_1, \hat{\tilde{G}}_2$ are fundamental operators for $\textbf{S}$ and $\textbf{S}^* = (S^*_1, S^*_2, S^*_3, \tilde{S}^*_1, \tilde{S}^*_2)$ respectively. Then for all $z \in \mathbb{D}$
		\begin{enumerate}
			\item $(G^*_1 + \tilde{G}_2z)\Theta_{S^*_3}(z) = \Theta_{S^*_3}(z)(\hat{G}_1 + \hat{\tilde{G}}^*_2z)$,
			
			\item $(G^*_2 + \tilde{G}_1z)\Theta_{S^*_3}(z) = \Theta_{S^*_3}(z)(\hat{G}_2 + \hat{\tilde{G}}^*_1z)$,
			
			\item $(\tilde{G}^*_1 + G_2z)\Theta_{S^*_3}(z) = \Theta_{S^*_3}(z)(\hat{\tilde{G}}_1 + \hat{G}^*_2z)$,
			
			\item $(\tilde{G}^*_2 + G_1z)\Theta_{S^*_3}(z) = \Theta_{S^*_3}(z)(\hat{\tilde{G}}_2 + \hat{G}^*_1z)$.
		\end{enumerate}
	\end{thm}

	We only state the following theorem as the proof is similar to that of Theorem \ref{Thm 2}.
	
	\begin{thm}\label{Thm 4}
		Let $\textbf{S} = (S_1, S_2, S_3, \tilde{S}_1, \tilde{S}_2)$ be a $\Gamma_{E(3; 2; 1, 2)}$-contraction on a Hilbert space $\mathcal{H}$. Suppose $G_1, 2G_2, 2\tilde{G}_1,\\ \tilde{G}_2$ and $\hat{G}_1, 2\hat{G}_2, 2\hat{\tilde{G}}_1, \hat{\tilde{G}}_2$ are fundamental operators for $\textbf{S}$ and $\textbf{S}^* = (S^*_1, S^*_2, S^*_3, \tilde{S}^*_1, \tilde{S}^*_2)$ respectively. Then for all $z \in \mathbb{D}$
		\begin{enumerate}\label{Admissible 2}
			\item\label{1} $(\hat{G}^*_1 + \hat{\tilde{G}}_2z)\Theta_{S_3}(z) = \Theta_{S_3}(z)(G_1 + \tilde{G}^*_2z)$,
			
			\item\label{2} $(\hat{G}^*_2 + \hat{\tilde{G}}_1z)\Theta_{S_3}(z) = \Theta_{S_3}(z)(G_2 + \tilde{G}^*_1z)$,
				
			\item\label{3} $(\hat{\tilde{G}}^*_1 + \hat{G}_2z)\Theta_{S_3}(z) = \Theta_{S_3}(z)(\tilde{G}_1 + G^*_2z)$,
				
			\item\label{4} $(\hat{\tilde{G}}^*_2 + \hat{G}_1z)\Theta_{S_3}(z) = \Theta_{S_3}(z)(\tilde{G}_2 + G^*_1z)$.
		\end{enumerate}
		
		Conversely, let $S_3$ be a pure contraction on a Hilbert space $\mathcal{H}$. Let $\hat{G}_1, 2\hat{G}_2, 2\hat{\tilde{G}}_1, \hat{\tilde{G}}_2 \in \mathcal{B}(\mathcal{D}_{S^*_3})$ with $w(\hat{G}^*_1 + \hat{\tilde{G}}_2z) \leqslant 1, w(\hat{G}^*_2 + \hat{\tilde{G}}_1z) \leqslant 1$ and satisfy
		\begin{equation}\label{Commutative 2(i)}
			\begin{aligned}
				&[\hat{G}_1,\hat{\tilde{G}}_i] = 0 \,\,\text{for}\,\, 1 \leq i \leq 2, [\hat{G}_2,\hat{\tilde{G}}_j]=0 \,\,\text{for}\,\, 1 \leq j \leq 2, \,\,\text{and}\,\, [\hat{G}_1, \hat{G}_2] = [\hat{\tilde{G}}_1, \hat{\tilde{G}}_2]  = 0,
			\end{aligned}
		\end{equation}
		$$\rm{and}$$
		\begin{equation}\label{Commutative 2(ii)}
			\begin{aligned}
				&[\hat{G}_1, \hat{G}^*_1] = [\hat{\tilde{G}}_2, \hat{\tilde{G}}^*_2], [\hat{G}_2, \hat{G}^*_2] = [\hat{\tilde{G}}_1, \hat{\tilde{G}}^*_1], [\hat{G}_1, \hat{\tilde{G}}^*_1] = [\hat{G}_2, \hat{\tilde{G}}^*_2],\\
				&[\hat{\tilde{G}}_1, \hat{G}^*_1] = [\hat{\tilde{G}}_2, \hat{G}^*_2], [\hat{G}_1, \hat{G}^*_2] = [\hat{\tilde{G}}_1, \hat{\tilde{G}}^*_2], [\hat{G}^*_1, \hat{G}_2] = [\hat{\tilde{G}}^*_1, \hat{\tilde{G}}_2].
			\end{aligned}
		\end{equation}
		Suppose $\hat{G}_1, 2\hat{G}_2, 2\hat{\tilde{G}}_1, \hat{\tilde{G}}_2$ satisfy $(1), (2), (3), (4)$ for some $G_1, 2G_2, 2\tilde{G}_1, \tilde{G}_2 \in \mathcal{B}(\mathcal{D}_{S_3})$ with $w(G^*_1 + \tilde{G}_2z) \leqslant 1, w(G^*_2 + \tilde{G}_1z) \leqslant 1$ then there exists a $\Gamma_{E(3; 2; 1, 2)}$-contraction $\textbf{S} = (S_1, S_2, S_3, \tilde{S}_1, \tilde{S}_2)$ such that $G_1, 2G_2, 2\tilde{G}_1, \tilde{G}_2$ be the fundamental operators of $\textbf{S}$ and $\hat{G}_1, 2\hat{G}_2, 2\hat{\tilde{G}}_1, \hat{\tilde{G}}_2$ be the fundamental operators of $\textbf{S}^* = (S^*_1, S^*_2, S^*_3, \tilde{S}^*_1, \tilde{S}^*_2)$.
	\end{thm}

	\begin{rem}\label{Rem 1}
		Note that the existence of $T_i$'s in Theorem \ref{Thm 2} is unique. Infact, if we assume that there exists $T_i$ and $T^{'}_i$ different operators for $1 \leqslant i \leqslant 6$ such that $F_1, \dots, F_6$ are the fundamental operators of $\textbf{T}$ and $\textbf{T}^{'}$; and $\tilde{F}_1, \dots, \tilde{F}_6$ are the fundamental operators of $\textbf{T}^*$ and $\textbf{T}^{'*}$. Then by [Theorem $3.2$, \cite{apal3}] we have that $T_i$ and $T^{'}_i$ are both unitarily equivalent to
		\begin{equation*}
			\begin{aligned}
				P_{\mathcal{H}_{T_7}}(I \otimes \tilde{F}^*_i + M_z \otimes \tilde{F}_{7-i})|_{\mathcal{H}_{T_7}} \,\,\text{for}\,\, 1 \leqslant i \leqslant 6,
			\end{aligned}
		\end{equation*}
		where $\mathcal{H}_{T_7}$ is $\Ran W$ defined as above. This implies that $T_i = T^{'}_i$ for $1 \leqslant i \leqslant 6$.
		
		
		Moreover, in the same way, the uniqueness of the $S_1, S_2, \tilde{S}_1, \tilde{S}_2$ holds true for the case of $\Gamma_{E(3; 2; 1, 2)}$-contraction.
	\end{rem}

	\section{Canonical Constructions of $\Gamma_{E(3; 3; 1, 1, 1)}$-Unitary and $\Gamma_{E(3; 2; 1, 2)}$-Unitary}\label{Canonical Construction}
	
	In this section, we construct $\Gamma_{E(3; 3; 1, 1, 1)}$-unitary from a $\Gamma_{E(3; 3; 1, 1, 1)}$-contraction. Similarly, we construct $\Gamma_{E(3; 2; 1, 2)}$-unitary from a $\Gamma_{E(3; 2; 1, 2)}$-contraction.
	
	\subsection{Construction of $\Gamma_{E(3; 3; 1, 1, 1)}$-Unitary from a $\Gamma_{E(3; 3; 1, 1, 1)}$-Contraction}\label{Canonical Construction 1}
	
	Let $\textbf{T} = (T_1, \dots, T_7)$ be a $\Gamma_{E(3; 3; 1, 1, 1)}$-contraction acting on a Hilbert space $\mathcal{H}$. Then $T_7$ is a contraction. Thus, there exists a positive semi-definite operator $Q_{T^*_7}$ such that
	\begin{equation}\label{Q_T^*_7}
		\begin{aligned}
			Q^2_{T^*_7} &= SOT-\lim_{n \to \infty} T^n_7T^{*n}_7.
		\end{aligned}
	\end{equation}
	We define an operator $V^{(7)*}_{T^*_7} : \overline{\Ran} Q_{T^*_7} \to \overline{\Ran} Q_{T^*_7}$ densely by
	\begin{equation}\label{V^7*_T^*_7}
		\begin{aligned}
			V^{(7)*}_{T^*_7}Q_{T^*_7}h &= Q_{T^*_7}T^*_7h \,\, \text{for $h \in \mathcal{H}$}.
		\end{aligned}
	\end{equation}
	Notice that
	\begin{equation}\label{V^*_T^* isometry}
		\begin{aligned}
			||V^{(7)*}_{T^*_7}Q_{T^*}h||^2
			&= \langle V^{(7)*}_{T^*_7}Q_{T^*_7}h, V^{(7)*}_{T^*_7}Q_{T^*_7}h \rangle\\
			&= \langle Q_{T^*_7}T^*_7h, Q_{T^*_7}T^*_7h \rangle\\
			&= \langle T_7Q^2_{T^*_7}T^*_7h, h \rangle\\
			&= \lim_{n \to \infty} \langle T^{n+1}_7T^{*n+1}_7h, h \rangle\\
			&= \langle Q^2_{T^*_7}h, h \rangle\\
			&= ||Q_{T^*_7}h||^2 \,\, \text{for all $h \in \mathcal{H}$}.
		\end{aligned}		
	\end{equation}
	Thus, $V^{(7)*}_{T^*_7}$ is an isometry on $\overline{\Ran} Q_{T^*_7}$.
	Since $T_i$'s are contractions, we have $T_iT^*_i \leqslant I$ for $1 \leqslant i \leqslant 6$. We define $V^{(i)*}_{T^*_7} : \overline{\Ran} Q_{T^*_7} \to \overline{\Ran} Q_{T^*_7}$ densely by
	\begin{equation}\label{V^i*_T^*_7}
		\begin{aligned}
			V^{(i)*}_{T^*_7}Q_{T^*_7}h
			&= Q_{T^*_7}T^*_ih \,\, \text{for $1 \leqslant i \leqslant 6$}.
		\end{aligned}
	\end{equation}
	It can be easily checked that the operators defined in \eqref{V^i*_T^*_7} are well-defined.
	The commutativity of $(V^{(1)*}_{T^*_7}, \dots, V^{(6)*}_{T^*_7},\\ V^{(7)*}_{T^*_7})$ readily follows from the commutativity of $T_i$'s and definition of $V^{(i)*}_{T^*_7}$.
	
	Now, we prove that $(V^{(1)*}_{T^*_7}, \dots, V^{(6)*}_{T^*_7}, V^{(7)*}_{T^*_7})$ is a $\Gamma_{E(3; 3; 1, 1, 1)}$-isometry. Notice that because $\textbf{T}$ is a $\Gamma_{E(3; 3; 1, 1, 1)}$-contraction then so is $\textbf{T}^*$. From \eqref{V^7*_T^*_7} and \eqref{V^i*_T^*_7} we have
	\begin{equation}\label{polynomial 1}
		\begin{aligned}
			p(V^{(1)*}_{T^*_7}, \dots, V^{(6)*}_{T^*_7}, V^{(7)*}_{T^*_7})h &= Q_{T^*_7}p(T^*_1, \dots, T^*_6, T^*_7)h
		\end{aligned}
	\end{equation}
	for any polynomial $p$ in $7$ variables and $h \in \mathcal{H}$. Thus it follows from \eqref{polynomial 1} that 
	\begin{equation*}
		\begin{aligned}
			||p(V^{(1)*}_{T^*_7}, \dots, V^{(6)*}_{T^*_7}, V^{(7)*}_{T^*_7})h||
			&= ||Q_{T^*_7}p(T^*_1, \dots, T^*_6, T^*_7)h||\\
			&\leqslant ||p(T^*_1, \dots, T^*_6, T^*_7)h||\\
			&\leqslant ||p(T^*_1, \dots, T^*_6, T^*_7)||\,||h||\\
			&\leqslant ||p||_{\infty, \Gamma_{E(3; 3; 1, 1, 1)}}\,||h||.
		\end{aligned}
	\end{equation*}
	Hence, $(V^{(1)*}_{T^*_7}, \dots, V^{(6)*}_{T^*_7}, V^{(7)*}_{T^*_7})$ is a $\Gamma_{E(3; 3; 1, 1, 1)}$-contraction. Since $V^{(7)*}_{T^*_7}$ is isometry, it follows from  [Theorem $4.4$, \cite{apal2}]  that $(V^{(1)*}_{T^*_7}, \dots, V^{(6)*}_{T^*_7}, V^{(7)*}_{T^*_7})$ is a $\Gamma_{E(3; 3; 1, 1, 1)}$-isometry. By definition of $\Gamma_{E(3; 3; 1, 1, 1)}$-isometry, $(V^{(1)*}_{T^*_7}, \dots, V^{(6)*}_{T^*_7}, V^{(7)*}_{T^*_7})$ can be extended to a $\Gamma_{E(3; 3; 1, 1, 1)}$-unitary $(N^{(1)*}_D, \dots, N^{(6)*}_D, N^{(7)*}_D)$ in a larger Hilbert space $\mathcal{Q}_{T^*_7} \supseteq \overline{\Ran}Q_{T^*_7}$, where $N^{(7)*}_D$ acting on $\mathcal{Q}_{T^*_7}$ is a minimal unitary dilation of $N^{(7)*}_{T^*_7}$. Therefore, $\textbf{N} = (N^{(1)}_D, \dots, N^{(6)}_D, N^{(7)}_D)$ is a $\Gamma_{E(3; 3; 1, 1, 1)}$-unitary on $\mathcal{Q}_{T^*_7}$ with $N^{(7)}_D$ is a minimal unitary dilation of $V^{(7)}_{T^*_7}$.
	
	\begin{defn}\label{Defn 3}
		Let $\textbf{T} = (T_1, \dots, T_7)$ be a $\Gamma_{E(3; 3; 1, 1, 1)}$-contraction on a Hilbert space $\mathcal{H}$ and $\textbf{N} = (N^{(1)}_D, \dots, N^{(6)}_D, N^{(7)}_D)$ be the $\Gamma_{E(3; 3; 1, 1, 1)}$-unitary constructed from $\textbf{T}$. We call $\textbf{N}$ the \textit{canonical \, $\Gamma_{E(3; 3; 1, 1, 1)}$-unitary associated to the $\Gamma_{E(3; 3; 1, 1, 1)}$-contraction $\textbf{T}$}.
	\end{defn}
	
	We prove that the canonical $\Gamma_{E(3; 3; 1, 1, 1)}$-unitary associated with a $\Gamma_{E(3; 3; 1, 1, 1)}$-contraction is unique up to unitary equivalence.
	
	\begin{thm}\label{Thm 5}
		Let $\textbf{T} = (T_1, \dots, T_7)$ on a Hilbert space $\mathcal{H}$ and $\textbf{T}^{'} = (T^{'}_1, \dots, T^{'}_7)$ on a Hilbert space $\mathcal{H}^{'}$. Let $\textbf{N} = (N^{(1)}_D, \dots, N^{(6)}_D, N^{(7)}_D)$ and $\textbf{N}^{'} = (N^{(1)'}_D, \dots, N^{(6)'}_D, N^{(7)'}_D)$ be the respective canonical $\Gamma_{E(3; 3; 1, 1, 1)}$-unitaries. If $\textbf{T}$ and $\textbf{T}^{'}$ are unitarily equivalent via $\tau$ then $\textbf{N}$ and $\textbf{N}^{'}$ are unitarily equivalent via the map $U_{\tau} : \mathcal{Q}_{T^*_7} \to \mathcal{Q}_{T^{'*}_7}$ defined by
		\begin{equation}\label{U_tau}
			\begin{aligned}
				U_{\tau}(N^{(7)n}_DQ_{T^*_7}h)
				&= N^{(7)'n}_DQ_{T^{'*}_7}\tau h
			\end{aligned}
		\end{equation}
		for all $n \geqslant 0$ and $h \in \mathcal{H}$.
	\end{thm}
	
	\begin{proof}
		Let $(V^{(1)*}_{T^*_7}, \dots, V^{(6)*}_{T^*_7}, V^{(7)*}_{T^*_7})$ and $(V^{(1)*}_{T^{'*}_7}, \dots, V^{(6)*}_{T^{'*}_7}, V^{(7)*}_{T^{'*}_7})$ be the $\Gamma_{E(3; 3; 1, 1, 1)}$-isometries constructed from $\textbf{T}$ and $\textbf{T}^{'}$ respectively as in the above construction. Suppose $\mathcal{Q}_{T^*_7}$ and $\mathcal{Q}_{T^{'*}_7}$ are the underlying Hilbert spaces of the canonical $\Gamma_{E(3; 3; 1, 1, 1)}$-unitaries $\textbf{N}$ and $\textbf{N}^{'}$ obtained from $\textbf{T}$ and $\textbf{T}^{'}$ respectively. Since, $T_7$ and $T^{'}_7$ are unitarily equivalent via $\tau$ then it implies that $\tau Q_{T^*_7} = Q_{T^{'*}_7}\tau$. Hence,
		\begin{equation}\label{tau, V^i*_T^*_7 commute}
			\begin{aligned}
				\tau V^{(i)*}_{T^*_7}Q_{T^*_7}h
				&= \tau Q_{T^*_7}T^*_ih
				= Q_{T^{'*}_7}\tau T^*_ih
				= Q_{T^{'*}_7}T^{'*}_i\tau h
				= V^{(i)*}_{T^{'*}_7}Q_{T^{'*}_7}\tau h
				= V^{(i)*}_{T^{'*}_7}\tau Q_{T^*_7}h.
			\end{aligned}
		\end{equation}
		Thus, we have $\tau V^{(i)*}_{T^*_7} = V^{(i)*}_{T^{'*}_7}\tau$ for $1 \leqslant i \leqslant 6$. Using the fact that $U_{\tau}|_{\overline{\Ran}Q_{T^*_7}} = \tau$, we get
		\begin{equation*}
			\begin{aligned}
				U_{\tau}N^{(7)}_DQ_{T^*_7}h
				&= N^{(7)'}_DQ_{T^{'*}_7}\tau h
				= N^{(7)'}_D\tau Q_{T^*_7}h
				= N^{(7)'}_DU_{\tau} Q_{T^*_7}h.
			\end{aligned}
		\end{equation*}
		
		To prove the unitary equivalence we proceed as follows:
		\begin{equation}\label{CCUE 1}
			\begin{aligned}
				U_{\tau}N^{(i)}_DN^{(7)n}_DQ_{T^*_7}h
				&= U_{\tau}N^{(7)n}_DN^{(i)}_DQ_{T^*_7}h\\
				&= U_{\tau}N^{(7)n+1}_DN^{(7)*}_DN^{(i)}_DQ_{T^*_7}h
			\end{aligned}
		\end{equation}
		Since, $\textbf{N}$ is a $\Gamma_{E(3; 3; 1, 1, 1)}$-unitary then by [Theorem $3.2$, \cite{apal2}] we have $N^{(7)*}_DN^{(i)}_D = N^{(7-i)*}_D$ and hence from \eqref{CCUE 1} we obtain
		\begin{equation}\label{CCUE 2}
			\begin{aligned}
				U_{\tau}N^{(i)}_DN^{(7)n}_DQ_{T^*_7}h&= N^{(7)'n+1}_DU_{\tau}N^{(7-i)*}_DQ_{T^*_7}h.
			\end{aligned}
		\end{equation}
		Observe that $U_{\tau}|_{\overline{\Ran}Q_{T^*_7}} = \tau, N^{(i)}_D|_{\overline{\Ran} Q_{T^*_7}} = V^{(i)}_{T^*_7}$ for $1 \leqslant i \leqslant 6$. Thus, from \eqref{CCUE 2} we get
		\begin{equation}\label{CCUE 3}
			\begin{aligned}
				U_{\tau}N^{(i)}_DN^{(7)n}_DQ_{T^*_7}h
				&= N^{(7)'n+1}_D\tau V^{(7-i)*}_{T^*_7}Q_{T^*_7}h\\		&= N^{(7)'n+1}_DV^{(7-i)*}_{T^{'*}_7}Q_{T^{'*}_7}\tau h \,\, (\text{by \eqref{tau, V^i*_T^*_7 commute}})\\
				&= N^{(7)'n+1}_DN^{(7-i)'*}_DQ_{T^{'*}_7}\tau h \,\, (\text{as $N^{(7-i)}_D|_{\overline{\Ran} Q_{T^*_7}} = V^{(7-i)}_{T^*_7}$}).
			\end{aligned}
		\end{equation}
		Because, $N^{(i)}_D$ and $N^{(7)}_D$ are commuting normal operators then from \textit{Fuglede's Theorem} [Theorem 1, \cite{Fuglede}], it follows that $N^{(7)'n+1}_DN^{(7-i)'*}_D = N^{(7-i)'*}_DN^{(7)'n+1}_D$ for $1 \leqslant i \leqslant 6$ and hence from \eqref{CCUE 3} we have the following
		\begin{equation}\label{CCUE 4}
			\begin{aligned}
				U_{\tau}N^{(i)}_DN^{(7)n}_DQ_{T^*_7}h&= N^{(7-i)'*}_DN^{(7)'n+1}_DQ_{T^{'*}_7}\tau h.
			\end{aligned}
		\end{equation}
		Because $\textbf{N}^{'}$ is a $\Gamma_{E(3; 3; 1, 1, 1)}$-unitary then by [Theorem $3.2$, \cite{apal2}] we have $N^{(7-i)'*}_DN^{(7)}_D = N^{(i)'*}_D$ and hence from \eqref{CCUE 4} it is immediate that
		\begin{equation}\label{CCUE 5}
			\begin{aligned}
				U_{\tau}N^{(i)}_DN^{(7)n}_DQ_{T^*_7}h&= N^{(i)'}_DN^{(7)'n}_D\tau Q_{T^*_7}h\\
				&= N^{(i)'}_DN^{(7)'n}_DU_{\tau}Q_{T^*_7}h \,\, (\text{as $U_{\tau}|_{\overline{\Ran}Q_{T^*_7}} = \tau$})\\
				&= N^{(i)'}_DU_{\tau}N^{(7)n}_DQ_{T^*_7}h.
			\end{aligned}
		\end{equation}
		As $N^{(7)*}_D$ is the minimal unitary dilation of $V^{(7)*}_{T^*_7}$ then $\{N^{(7)n}_DQ_{T^*_7}h : n \geqslant 0, h \in \mathcal{H}\}$ is dense in $\mathcal{Q}_{T^*_7}$ then $N^{(i)}_D$ and $N^{(i)'}_D$ are unitarily equivalent and therefore, $\textbf{N}$ is unitarily equivalent to $\textbf{N}^{'}$ via the map $U_{\tau}$. This completes the proof.
	\end{proof}
	
	\subsection{Construction of $\Gamma_{E(3; 2; 1, 2)}$-Unitary from a $\Gamma_{E(3; 2; 1, 2)}$-Contraction}\label{Canonical Construction 2}
	
	Suppose that $\textbf{S} = (S_1, S_2, S_3, \tilde{S}_1, \tilde{S}_2)$ be a $\Gamma_{E(3; 2; 1, 2)}$-contraction defined on a Hilbert space $\mathcal{H}$. Then $S_3$ is a contraction. Thus, there exists a positive semi-definite operator $Q_{S^*_3}$ such that
	\begin{equation}\label{Q_S^*_3}
		\begin{aligned}
			Q^2_{S^*_3} &= SOT-\lim_{n \to \infty} S^n_3S^{*n}_3.
		\end{aligned}
	\end{equation}
	We consider an operator $W^{(3)*}_{S^*_3} : \overline{\Ran} Q_{S^*_3} \to \overline{\Ran} Q_{S^*_3}$ densely defined by
	\begin{equation}\label{W^{(3)*}_S^*_3}
		\begin{aligned}
			W^{(3)*}_{S^*_3}Q_{S^*_3}h &= Q_{S^*_3}S^*_3h.
		\end{aligned}
	\end{equation}
	By the similar argument given in \eqref{V^*_T^* isometry} we have that $W^{(3)*}_{S^*_3}$ is an isometry on $\overline{\Ran} Q_{S^*_3}$. As $\textbf{S}$ is a $\Gamma_{E(3; 2; 1, 2)}$-contraction then $S_1, \frac{S_2}{2}, \frac{\tilde{S}_1}{2}, \tilde{S}_2$ are contractions. We define $W^{(1)*}_{S^*_3}, W^{(2)*}_{S^*_3}, \tilde{W}^{(1)*}_{S^*_3}, \tilde{W}^{(2)*}_{S^*_3} : \overline{\Ran} Q_{S^*_3} \to \overline{\Ran} Q_{S^*_3}$ densely as follows:
	\begin{equation}\label{W^i*_S^*_3, tilda W^j*_S^*_3}
		\begin{aligned}
			W^{(i)*}_{S^*_3}Q_{S^*_3}h
			&= Q_{S^*_3}S^*_ih \,\, \text{and} \,\, \tilde{W}^{(j)*}_{S^*_3}Q_{S^*_3}h = Q_{S^*_3}\tilde{S}^*_jh \,\, \text{for $1 \leqslant i, j \leqslant 2$}.
		\end{aligned}
	\end{equation}
	By the commutativity of $S_1, S_2, S_2, \tilde{S}_1$ and $\tilde{S}_2$ it can be deduced that $(W^{(1)*}_{S^*_3}, W^{(2)*}_{S^*_3}, W^{(3)*}_{S^*_3}, \tilde{W}^{(1)*}_{S^*_3}, \tilde{W}^{(2)*}_{S^*_3})$ is a commuting tuple of operators.
	
	We show that $(W^{(1)*}_{S^*_3}, W^{(2)*}_{S^*_3}, W^{(3)*}_{S^*_3}, \tilde{W}^{(1)*}_{S^*_3}, \tilde{W}^{(2)*}_{S^*_3})$ is a $\Gamma_{E(3; 2; 1, 2)}$-contraction. For any polynomial $f$ in $5$ variables and for all $h \in \mathcal{H}$ we have
	\begin{equation*}
		\begin{aligned}
			 ||f(W^{(1)*}_{S^*_3}, W^{(2)*}_{S^*_3}, W^{(3)*}_{S^*_3}, \tilde{W}^{(1)*}_{S^*_3}, \tilde{W}^{(2)*}_{S^*_3})h||
			 &= ||Q_{S^*_3}f(S^*_1, S^*_2, S^*_3, \tilde{S}^*_1, \tilde{S}^*_2)h||\\
			 &\leqslant ||f(S^*_1, S^*_2, S^*_3, \tilde{S}^*_1, \tilde{S}^*_2)h||\\
			 &\leqslant ||f(S^*_1, S^*_2, S^*_3, \tilde{S}^*_1, \tilde{S}^*_2)||\,||h||\\
			 &\leqslant ||f||_{\infty, \Gamma_{E(3; 2; 1, 2)}}\,||h||.
		\end{aligned}
	\end{equation*}
	This implies that $(W^{(1)*}_{S^*_3}, W^{(2)*}_{S^*_3}, W^{(3)*}_{S^*_3}, \tilde{W}^{(1)*}_{S^*_3}, \tilde{W}^{(2)*}_{S^*_3})$ is a $\Gamma_{E(3; 2; 1, 2)}$-contraction; and as $(W^{(1)*}_{S^*_3}, W^{(2)*}_{S^*_3}, W^{(3)*}_{S^*_3},\\ \tilde{W}^{(1)*}_{S^*_3}, \tilde{W}^{(2)*}_{S^*_3})$ an isometry we conclude by [Theorem $3.7$, \cite{apal2}] that $\textbf{W}^*$ is a $\Gamma_{E(3; 2; 1, 2)}$-isometry on $\overline{\Ran} Q_{S^*_3}$. Therefore, by definition of $\Gamma_{E(3; 2; 1, 2)}$-isometry, it can be dilated to a $\Gamma_{E(3; 2; 1, 2)}$-unitary $(M^{(1)*}_D, M^{(2)*}_D, M^{(3)*}_D,\\ \tilde{M}^{(1)*}_D, \tilde{M}^{(2)*}_D)$ in a larger Hilbert space $\mathcal{Q}_{S^*_3}$ containing $\overline{\Ran} Q_{S^*_3}$, where $M^{(3)*}_D$ is operating on $\mathcal{Q}_{S^*_3}$ is a minimal unitary dilation of $W^{(3)*}_{S^*_3}$ and therefore $\textbf{M} = (M^{(1)}_D, M^{(2)}_D, M^{(3)}_D, \tilde{M}^{(1)}_D, \tilde{M}^{(2)}_D)$ is a $\Gamma_{E(3; 2; 1, 2)}$-unitary on $\mathcal{Q}_{S^*_3}$.
	
	\begin{defn}\label{Defn 4}
		Let $\textbf{S} = (S_1, S_3, S_3, \tilde{S}_1, \tilde{S}_2)$ be a $\Gamma_{E(3; 2; 1, 2)}$-contraction on a Hilbert space $\mathcal{H}$ and $\textbf{M} = (M^{(1)}_D, M^{(2)}_D, M^{(3)}_D, \tilde{M}^{(1)}_D, \tilde{M}^{(2)}_D)$ be the $\Gamma_{E(3; 2; 1, 2)}$-unitary constructed from $\textbf{S}$. We call $\textbf{M}$ the \textit{canonical $\Gamma_{E(3; 2; 1, 2)}$-unitary associated to the $\Gamma_{E(3; 2; 1, 2)}$-contraction $\textbf{S}$}.
	\end{defn}
	
	We only state the following theorem as the proof is similar to Theorem \ref{Thm 5}.
	
	\begin{thm}\label{Thm 6}
		Let $\textbf{S} = (S_1, S_3, S_3, \tilde{S}_1, \tilde{S}_2)$ on a Hilbert space $\mathcal{H}$ and $\textbf{S}^{'} = (S^{'}_1, S^{'}_3, S^{'}_3, \tilde{S}^{'}_1, \tilde{S}^{'}_2)$ on a Hilbert space $\mathcal{H}^{'}$. Let $\textbf{M} = (M^{(1)}_D, M^{(2)}_D, M^{(3)}_D, \tilde{M}^{(1)}_D, \tilde{M}^{(2)}_D)$ and $\textbf{M}^{'} = (M^{(1)'}_D, M^{(2)'}_D, M^{(3)'}_D, \tilde{M}^{(1)'}_D, \tilde{M}^{(2)'}_D)$ be the respective canonical $\Gamma_{E(3; 2; 1, 2)}$-unitaries. If $\textbf{S}$ and $\textbf{S}^{'}$ are unitarily equivalent via the map $\sigma$ then $\textbf{M}$ and $\textbf{M}^{'}$ are unitarily equivalent via the map $U_{\sigma} : \mathcal{Q}_{S^*_3} \to \mathcal{Q}_{S^{'*}_3}$ defined by
		\begin{equation}\label{U_sigma}
			\begin{aligned}
				U_{\sigma}(M^{(3)n}_DQ_{S^*_3}h)
				&= M^{(3)'n}_DQ_{S^{'*}_3}\sigma h
			\end{aligned}
		\end{equation}
		for all $n \geqslant 0$ and $h \in \mathcal{H}$.
	\end{thm}
	
	\section{Models for $\Gamma_{E(3; 3; 1, 1, 1)}$-Isometry and $\Gamma_{E(3; 2; 1, 2)}$-Isometry}\label{Model for Isometries}
	
	In this section, we develop models for $\Gamma_{E(3; 3; 1, 1, 1)}$-isometry and $\Gamma_{E(3; 2; 1, 2)}$-isometry. Let $T$ be an isometry on a Hilbert space $\mathcal{H}$, then by \textit{von Neumann-Wold decomposition} we have that there exists Hilbert spaces $\mathcal{E}, \mathcal{F}$ and a unitary
	$U : \mathcal{H} \to
	\begin{pmatrix}
		H^2(\mathcal{E})\\
		\mathcal{F}
	\end{pmatrix}$
	such that
	\begin{equation}\label{UTU*}
		\begin{aligned}
			UTU^*
			&=
			\begin{pmatrix}
				M_z & 0\\
				0 & N_D
			\end{pmatrix} :
			\begin{pmatrix}
				H^2(\mathcal{E})\\
				\mathcal{F}
			\end{pmatrix} \to
			\begin{pmatrix}
				H^2(\mathcal{E})\\
				\mathcal{F}
			\end{pmatrix}
		\end{aligned}
	\end{equation}
	where $M_z$ is the unilateral shift on $H^2(\mathcal{E})$ and $N_D$ is a unitary acting on $\mathcal{F}$. We show that a $7$-tuple (respectively, $5$-tuple) of commuting bounded operators $\textbf{T} = (T_1, \dots, T_7)$ (respectively, $\textbf{S} = (S_1, S_2, S_3, \tilde{S}_1, \tilde{S}_2)$) is a $\Gamma_{E(3; 3; 1, 1, 1)}$-isometry (respectively, $\Gamma_{E(3; 2; 1, 2)}$-isometry) if and only if it possesses von Neumann-Wold decomposition. We first see some examples of $\Gamma_{E(3; 3; 1, 1, 1)}$-isometry and $\Gamma_{E(3; 2; 1, 2)}$-isometry.
	
	\begin{exam}\label{Example 1}
		Let $\mathcal{E}$ be a Hilbert space and $H^2(\mathcal{E})$ be the Hardy space of $\mathcal{E}$-valued functions and $F_1, \dots, F_6 \in \mathcal{B}(\mathcal{E})$ that satisfy the conditions
		\begin{equation}\label{Commutative A}
			\begin{aligned}
				&[F_i, F_j] = 0 \,\, \text{and} \,\, [F^*_i, F_{7-j}] = [F^*_j, F_{7-i}]
			\end{aligned}
		\end{equation}
		$$\rm{and}$$
		\begin{equation}\label{Norm condition 1}
			\begin{aligned}
				||F^*_i + F_{7-i}z||_{\infty, \mathbb{T}} &\leqslant 1
			\end{aligned}
		\end{equation}
		for $1 \leqslant i, j \leqslant 6$. Consider the operators
		\begin{equation*}
			\begin{aligned}
				T_i = M_{F^*_i + F_{7-i}z} \,\, \text{for} \,\, 1 \leqslant i \leqslant 6 \,\, \textit{and} \,\, T_7 = M_z \,\, \text{on} \,\, H^2(\mathcal{E}).
			\end{aligned}
		\end{equation*}
		Then $T_7$ commutes with $T_i$. Again it follows from \eqref{Commutative A} that $T_iT_j = T_jT_i$ for $1 \leqslant i, j \leqslant 6$. Notice that
		\begin{equation*}
			\begin{aligned}
				T^*_{7-i}T_7
				&= M^*_{F^*_{7-i} + F_iz}M_z\\
				&= (I_{H^2} \otimes F_{7-i} + M^*_z \otimes F^*_i)(M_z \otimes I_{\mathcal{E}})\\
				&= M_z \otimes F_{7-i} + I_{H^2} \otimes F^*_i\\
				&= M_{F^*_i + F_{7-i}z}\\
				&= T_i.
			\end{aligned}
		\end{equation*}
		Thus, $T_i = T^*_{7-i}T_7$ for $1 \leqslant i \leqslant 6$. From \eqref{Norm condition 1} we have $M_{F^*_i + F_{7-i}z}$ are contractions for $1 \leqslant i \leqslant 6$. Therefore, by [Theorem $4.4$, \cite{apal2}] we have $(M_{F^*_1 + F_6z}, \dots, M_{F^*_6 + F_1z}, M_z)$ is a $\Gamma_{E(3; 3; 1, 1, 1)}$-isometry on $H^2(\mathcal{E})$. Also by [Theorem $4.6$, \cite{apal2}], any pure $\Gamma_{E(3; 3; 1, 1, 1)}$-isometry is unitarily equivalent to a $\Gamma_{E(3; 3; 1, 1, 1)}$-isometry of this form.
	\end{exam}
	
	\begin{exam}\label{Example 2}
		Let $\mathcal{E}$ and $H^2(\mathcal{E})$ are as in Example \ref{Example 1}. Take $G_1, 2G_2, 2\tilde{G}_1, \tilde{G}_2 \in \mathcal{B}(\mathcal{E})$ such that
		\begin{equation}\label{Commutative B(i)}
			\begin{aligned}
				&[G_1, \tilde{G}_i] = 0 \,\,\text{for}\,\, 1 \leqslant i \leq 2, [G_2, \tilde{G}_j]=0 \,\,\text{for}\,\, 1 \leqslant j \leqslant 2, \,\,\text{and}\,\, [G_1, G_2] = [\tilde{G}_1, \tilde{G}_2]  = 0,
			\end{aligned}
		\end{equation}
		$$\rm{and}$$
		\begin{equation}\label{Commutative B(ii)}
			\begin{aligned}
				&[G_1, G^*_1] = [\tilde{G}_2, \tilde{G}^*_2], [G_2, G^*_2] = [\tilde{G}_1, \tilde{G}^*_1], [G_1, \tilde{G}^*_1] = [G_2, \tilde{G}^*_2],\\
				&[\tilde{G}_1, G^*_1] = [\tilde{G}_2, G^*_2], [G_1, G^*_2] = [\tilde{G}_1, \tilde{G}^*_2], [G^*_1, G_2] = [\tilde{G}^*_1, \tilde{G}_2].
			\end{aligned}
		\end{equation}
		Suppose $G_1, 2G_2, 2\tilde{G}_1, \tilde{G}_2$ satisfy
		\begin{equation}\label{Norm condition 2}
			\begin{aligned}
				&||G^*_1 + \tilde{G}_2z||_{\infty, \mathbb{T}} \leqslant 1 \,\, \text{and} \,\, ||G^*_2 + \tilde{G}_1z||_{\infty, \mathbb{T}} \leqslant 1.
			\end{aligned}
		\end{equation}
		Let us consider the following operators:
		\begin{equation*}
			\begin{aligned}
				&S_1 = M_{G^*_1 + \tilde{G}_2z}, S_2 = M_{2G^*_2 + 2\tilde{G}_1z}, S_3 = M_z, \tilde{S}_1 = M_{2\tilde{G}^*_1 + 2G_2z}, \tilde{S}_2 = M_{\tilde{G}^*_2 + G_1z} \,\, \text{on} \,\, H^2(\mathcal{E}).
			\end{aligned}
		\end{equation*}
		One can easily verify that $S_1 = \tilde{S}^*_2S_3$ and $S_2 = \tilde{S}^*_1S_3$. It yields from \eqref{Norm condition 2} that $||S_1|| \leqslant 1, ||S_2|| \leqslant 2, ||\tilde{S}_1|| \leqslant 2$ and $||\tilde{S}_2|| \leqslant 1$. Therefore, by [Theorem $4.5$, \cite{apal2}], $(S_1, S_2, S_3, \tilde{S}_1, \tilde{S}_2)$ is a $\Gamma_{E(3; 2; 1, 2)}$-isometry and by [Theorem $4.7$, \cite{apal2}] we have that any pure $\Gamma_{E(3; 2; 1, 2)}$-isometry is unitarily equivalent to a $\Gamma_{E(3; 2; 1, 2)}$-isometry of this form.
	\end{exam}
	
	Proof of the following lemma is straight forward and it can be found in \cite{Foias}. We thus omit the proof.
	
	\begin{lem}\label{Lem 3}
		Let $V$ be a unitary operator on $\mathcal{H}_2$ and $T$ be an operator on $\mathcal{H}_1$ such that $T^{*n}\to 0$ in the strong operator topology as $n \to \infty$. If $X$ is a bounded operator from
		$\mathcal{H}_2$ to $\mathcal{H}_1$ such that $XV = TX$ then $X = 0$.
	\end{lem}
	
	Now we will proceed for the main results of this section.
	
	\begin{thm}[Model for $\Gamma_{E(3; 3; 1, 1, 1)}$-Isometry]\label{Wold decomposition 1}
		Let $\textbf{T} = (T_1, \dots, T_7)$ be a commuting $7$-tuple of bounded operators on a Hilbert space $\mathcal{H}$. Then $\textbf{T}$ is a $\Gamma_{E(3; 3; 1, 1, 1)}$-isometry if and only if there exists Hilbert spaces $\mathcal{E}, \mathcal{F}$ such that $\mathcal{H}$ is isomorphic to
		$\begin{pmatrix}
			H^2(\mathcal{E})\\
			\mathcal{F}
		\end{pmatrix}$ and with respect to the same unitary, $\textbf{T}$ is unitarily equivalent to
		\begin{equation}\label{Model, WD 1}
			\begin{aligned}
				\left(
				\begin{pmatrix}
					M_{F^*_1 + F_6z} & 0\\
					0 & N^{(1)}_D
				\end{pmatrix}, \dots,
				\begin{pmatrix}
					M_{F^*_6 + F_1z} & 0\\
					0 & N^{(6)}_D
				\end{pmatrix},
				\begin{pmatrix}
					M_z & 0\\
					0 & N^{(7)}_D
				\end{pmatrix}
				\right)
			\end{aligned}
		\end{equation}
		acting on
		$\begin{pmatrix}
			H^2(\mathcal{E})\\
			\mathcal{F}
		\end{pmatrix}$
		for some $F_1, \dots, F_6 \in \mathcal{B}(\mathcal{E})$ satisfying \eqref{Commutative A} and \eqref{Norm condition 1} and $(N^{(1)}_D, \dots, N^{(6)}_D, N^{(7)}_D)$ is a $\Gamma_{E(3; 3; 1, 1, 1)}$-unitary acting on $\mathcal{F}$.
	\end{thm}
	
	\begin{proof}
		The sufficiency part is immediate from Example \ref{Example 1}.
		
		In order to prove the other direction, let $\textbf{T}$ be a $\Gamma_{E(3; 3; 1, 1, 1)}$-isometry. Thus by [Theorem $4.4$, \cite{apal2}], $T_7$ is an isometry. Then by von Neumann-Wold decomposition of $T_7$, there exists Hilbert spaces $\mathcal{E}, \mathcal{F}$ and a unitary $U : \mathcal{H} \to
		\begin{pmatrix}
			H^2(\mathcal{E})\\
			\mathcal{F}
		\end{pmatrix}$
		such that
		\begin{equation*}
			\begin{aligned}
				UT_7U^*
				&=
				\begin{pmatrix}
					M_z & 0\\
					0 & N^{(7)}_D
				\end{pmatrix} :
				\begin{pmatrix}
					H^2(\mathcal{E})\\
					\mathcal{F}
				\end{pmatrix} \to
				\begin{pmatrix}
					H^2(\mathcal{E})\\
					\mathcal{F}
				\end{pmatrix}
			\end{aligned}
		\end{equation*}
		for some unitary $N^{(7)}_D$ acting on $\mathcal{F}$. Assume that
		\begin{equation*}
			\begin{aligned}
				UT_iU^*
				&=
				\begin{pmatrix}
					A^{(i)}_{11} & A^{(i)}_{12}\\
					A^{(i)}_{21} & A^{(i)}_{22}
				\end{pmatrix} :
				\begin{pmatrix}
					H^2(\mathcal{E})\\
					\mathcal{F}
				\end{pmatrix} \to
				\begin{pmatrix}
					H^2(\mathcal{E})\\
					\mathcal{F}
				\end{pmatrix} \,\, \text{for $1 \leqslant i \leqslant 6$}.
			\end{aligned}
		\end{equation*}
		Since $T_i$ commutes with $T_7$ then $UT_iU^*$ commutes with $UT_7U^*$ as well. Thus, we have the following:
		\begin{equation}\label{M 1.1}
			\begin{aligned}
				&A^{(i)}_{11}M_z = M_zA^{(i)}_{11}, A^{(i)}_{12}N^{(7)}_D = M_zA^{(i)}_{12}, A^{(i)}_{21}M_z = N^{(7)}_DA^{(i)}_{21}, A^{(i)}_{2}N^{(7)}_D = N^{(7)}_DA^{(i)}_{22} \,\, \text{for $1 \leqslant i \leqslant 6$}.
			\end{aligned}
		\end{equation}
		It now follows from Lemma \ref{Lem 3} that $A^{(i)}_{21} = 0$ for $1 \leqslant i \leqslant 6$. As $\textbf{T}$ is a $\Gamma_{E(3; 3; 1, 1, 1)}$-isometry then by [Theorem $4.4$, \cite{apal2}] we have $T_i = T^*_{7-i}T_7$ for $1 \leqslant i \leqslant 6$. Observe that
		\begin{equation}\label{M 1.2}
			\begin{aligned}
				\begin{pmatrix}
					A^{(i)}_{11} & A^{(i)}_{12}\\
					0 & A^{(i)}_{22}
				\end{pmatrix}
				= T_i
				&= T^*_{7-i}T_7\\
				&=
				\begin{pmatrix}
					A^{(7-i)*}_{11} & 0\\
					A^{(7-i)*}_{12} & A^{(7-i)*}_{22}
				\end{pmatrix}
				\begin{pmatrix}
					M_z & 0\\
					0 & N_D
				\end{pmatrix}\\
				&=
				\begin{pmatrix}
					A^{(7-i)*}_{11}M_z & 0\\
					A^{(7-i)*}_{12}M_z & A^{(7-i)*}_{22}N_D
				\end{pmatrix}.
			\end{aligned}
		\end{equation}
		From \eqref{M 1.2} we see that $A^{(i)}_{12} = 0$ for $1 \leqslant i \leqslant 6$. Since $A^{(i)}_{11}$ commutes with $M_z$ then there exists $\Phi_i \in H^{\infty}(\mathcal{E})$ such that $X_i = M_{\Phi_i}$. This implies that $M_{\Phi_i} = M^*_{\Phi_{7-i}}M_z$ and hence we obtain from here that $\Phi_i(z) = z\Phi^*_{7-i}(z)$ for all $1 \leqslant i \leqslant 6$ and $z \in \mathbb{T}$. Let the power series expansion of $\Psi_i$ be
		\begin{equation*}
			\begin{aligned}
				\Phi_i(z) &= \sum_{n \geqslant 0} C^{(i)}_nz^n \,\, \text{for} \,\, 1 \leqslant i \leqslant 6 \,\, \text{and} \,\, z \in \mathbb{T}.
			\end{aligned}
		\end{equation*}
		Then from $\Phi_i(z) = z\Phi^*_{7-i}(z)$ we obtain the following:
		\begin{equation}\label{M 1.3}
			\begin{aligned}
				C^{(i)}_0 + C^{(i)}_1z + \sum_{n \geqslant 2} C^{(i)}_nz^n &= z\sum_{n \geqslant 0}C^{(7-i)*}_n\overline{z}^n = C^{(7-i)*}_0z + C^{(7-i)*}_1 + \sum_{n \geqslant 2} C^{(7-i)*}_n\overline{z}^{n-1} 
			\end{aligned}
		\end{equation}
		for all $z \in \mathbb{T}$. Thus comparing the coefficients of $z^n, \overline{z}^n$ for $n \geqslant 2$ and the constant terms we obtain
		\begin{equation}\label{M 1.4}
			\begin{aligned}
				C^{(i)}_0 = C^{(7-i)}_1, C^{(i)}_1 = C^{(7-i)}_0.
			\end{aligned}
		\end{equation}
		It follows from here that $\Phi_i$'s are of the form
		\begin{equation}\label{Phi_i}
			\begin{aligned}
				\Phi_i(z) &= F^*_i + F_{7-i}z
			\end{aligned}
		\end{equation}
		for some $F_i \in \mathcal{B}(\mathcal{D}_{T_7})$ for $1 \leqslant i \leqslant 6$ and $z \in \mathbb{T}$.
		
		Therefore, it follows from \eqref{Phi_i} that $UT_iU^*$ are of the form
		\begin{equation*}
			\begin{aligned}
				UT_iU^*
				&=
				\begin{pmatrix}
					M_{F^*_i + F_{7-i}z} & 0\\
					0 & N^{(i)}_D
				\end{pmatrix} \,\, \text{for $1 \leqslant i \leqslant 6$}
			\end{aligned}
		\end{equation*}
		Hence, by [Theorem $4.6$, \cite{apal2}], $(M_{F^*_1 + F_6z}, \dots, M_{F^*_6 + F_1z}, M_z)$ is a pure $\Gamma_{E(3; 3; 1, 1, 1)}$-isometry. On the other hand, $(N^{(1)}_D, \dots, N^{(6)}_D, N^{(7)}_D)$ is a $\Gamma_{E(3; 3; 1, 1, 1)}$-contraction with $N^{(7)}_D$ unitary on $\mathcal{F}$. Therefore, by [Theorem $3.2$, \cite{apal2}], $(N^{(1)}_D, \dots, N^{(6)}_D, N^{(7)}_D)$ is a $\Gamma_{E(3; 3; 1, 1, 1)}$-unitary acting on $\mathcal{F}$. This completes the proof.
	\end{proof}
	
	The following result is analogous for $\Gamma_{E(3; 2; 1, 2)}$-isometry. We only state the theorem as the proof is similar to that of Theorem \ref{Wold decomposition 1}.
	
	\begin{thm}[Model for $\Gamma_{E(3; 2; 1, 2)}$-Isometry]\label{Wold decomposition 2}
		Let $\textbf{S} = (S_1, S_2, S_3, \tilde{S}_1, \tilde{S}_2)$ be a commuting $5$-tuple of bounded operators on a Hilbert space $\mathcal{H}$. Then $\textbf{S}$ is a $\Gamma_{E(3; 2; 1, 2)}$-isometry if and only if there exists Hilbert spaces $\mathcal{E}, \mathcal{F}$ such that $\mathcal{H}$ is isomorphic to
		$\begin{pmatrix}
			H^2(\mathcal{E})\\
			\mathcal{F}
		\end{pmatrix}$ and with respect to the same unitary, $\textbf{S}$ is unitarily equivalent to
		\begin{equation}\label{Model, WD 2}
			\begin{aligned}
				&\left(
				\begin{pmatrix}
					M_{G^*_1 + \tilde{G}_2z} & 0\\
					0 & M^{(1)}_D
				\end{pmatrix},
				\begin{pmatrix}
					M_{2G^*_2 + 2\tilde{G}_1z} & 0\\
					0 & M^{(2)}_D
				\end{pmatrix},
				\begin{pmatrix}
					M_z & 0\\
					0 & M^{(3)}_D
				\end{pmatrix}, \right.\\
				&\left. \hspace{2.5cm}
				\begin{pmatrix}
					M_{2\tilde{G}^*_1 +2 G_2z} & 0\\
					0 & \tilde{M}^{(1)}_D
				\end{pmatrix}
				\begin{pmatrix}
					M_{\tilde{G}^*_2 + G_1z} & 0\\
					0 & \tilde{M}^{(2)}_D
				\end{pmatrix}
				\right)
			\end{aligned}
		\end{equation}
		acting on
		$\begin{pmatrix}
			H^2(\mathcal{E})\\
			\mathcal{F}
		\end{pmatrix}$
		for some bounded operators $G_1, 2G_2, 2\tilde{G}_1, \tilde{G}_2$ defined on $\mathcal{E}$ satisfying \eqref{Commutative B(i)}, \eqref{Commutative B(ii)}, \eqref{Norm condition 2} and $(M^{(1)}_D, M^{(2)}_D, M^{(3)}_D, \tilde{M}^{(1)}_D, \tilde{M}^{(2)}_D)$ is a $\Gamma_{E(3; 2; 1, 2)}$-unitary acting on $\mathcal{F}$.
	\end{thm}
	
	In one variable case, any isometry can always be extended to a unitary. The following results present an analogous version for $\Gamma_{E(3; 3; 1, 1, 1)}$-isometry (respectively, $\Gamma_{E(3; 2; 1, 2)}$-isometry). Those results indicate that any $\Gamma_{E(3; 3; 1, 1, 1)}$-isometry (respectively, $\Gamma_{E(3; 2; 1, 2)}$-isometry) can be extended to a $\Gamma_{E(3; 3; 1, 1, 1)}$-unitary (respectively, $\Gamma_{E(3; 2; 1, 2)}$-unitary).
	
	\begin{cor}\label{Cor 1}
		Let $\textbf{T} = (T_1, \dots, T_7)$ be a commuting $7$-tuple of bounded operators on a Hilbert space $\mathcal{H}$. Then $\textbf{T}$ is a $\Gamma_{E(3; 3; 1, 1, 1)}$-isometry if and only if it can be extended to a $\Gamma_{E(3; 3; 1, 1, 1)}$-unitary acting on a Hilbert space of minimal unitary dilation of the isometry $T_7$.
	\end{cor}
	
	\begin{proof}
		Let $\textbf{T}$ can be extended to a $\Gamma_{E(3; 3; 1, 1, 1)}$-unitary acting on Hilbert space of minimal unitary dilation of the isometry $T_7$. Then it is clear that $\textbf{T}$ is the restriction of a $\Gamma_{E(3; 3; 1, 1, 1)}$-unitary to a joint invariant subspace. Therefore $\textbf{T}$ is a $\Gamma_{E(3; 3; 1, 1, 1)}$-isometry.
		
		Conversely, let $\textbf{T}$ is a $\Gamma_{E(3; 3; 1, 1, 1)}$-isometry. Without loss of generality, let us consider $\textbf{T}$ be
		\begin{equation*}
			\begin{aligned}
				\left(
				\begin{pmatrix}
					M_{F^*_1 + F_6z} & 0\\
					0 & N^{(1)}_D
				\end{pmatrix}, \dots,
				\begin{pmatrix}
					M_{F^*_6 + F_1z} & 0\\
					0 & N^{(6)}_D
				\end{pmatrix},
				\begin{pmatrix}
					M_z & 0\\
					0 & N^{(7)}_D
				\end{pmatrix}
				\right)
			\end{aligned}
		\end{equation*}
		acting on
		$
		\begin{pmatrix}
			H^2(\mathcal{E})\\
			\mathcal{F}
		\end{pmatrix}$
		for some $F_1, \dots, F_6 \in \mathcal{B}(\mathcal{E})$ satisfying \eqref{Commutative A} and \eqref{Norm condition 1} and $(N^{(1)}_D, \dots, N^{(6)}_D, N^{(7)}_D)$ is a $\Gamma_{E(3; 3; 1, 1, 1)}$-unitary acting on $\mathcal{F}$.
		We consider $H^2(\mathcal{E})$ as a closed subspace of $L^2(\mathcal{E})$. Then $7$-tuple of operators
		\begin{equation}\label{C 1.1}
			\begin{aligned}
				\left(
				\begin{pmatrix}
					M_{F^*_1 + F_6\omega} & 0\\
					0 & N^{(1)}_D
				\end{pmatrix}, \dots,
				\begin{pmatrix}
					M_{F^*_6 + F_1\omega} & 0\\
					0 & N^{(6)}_D
				\end{pmatrix},
				\begin{pmatrix}
					M_\omega & 0\\
					0 & N^{(7)}_D
				\end{pmatrix}
				\right) :
				\begin{pmatrix}
					L^2(\mathcal{E})\\
					\mathcal{F}
				\end{pmatrix} \to
				\begin{pmatrix}
					L^2(\mathcal{E})\\
					\mathcal{F}
				\end{pmatrix}.
			\end{aligned}
		\end{equation}
		is an extension of $\textbf{T}$. Observe that
		$
		\begin{pmatrix}
			M_\omega & 0\\
			0 & N^{(7)}_D
		\end{pmatrix}$
		is a minimal unitary dilation of
		$
		\begin{pmatrix}
			M_z & 0\\
			0 & N^{(7)}_D
		\end{pmatrix}$.
		Also note that since $(M_{F^*_1 + F_6z}, \dots, M_{F^*_6 + F_1z}, M_z)$ is commutative then
		\begin{equation}\label{C 1.2}
			\begin{aligned}
				M_{F^*_i + F_{7-i}z}M_{F^*_j + F_{7-j}z} &=
				M_{F^*_j + F_{7-j}z}M_{F^*_i + F_{7-i}z}
			\end{aligned}
		\end{equation}
		for all $z \in \mathbb{T}$ and $1 \leqslant i, j \leqslant 6$. By \eqref{C 1.2} we get
		\begin{equation}\label{C 1.3}
			\begin{aligned}
				(F^*_i + F_{7-i}z)(F^*_j + F_{7-j}z) &=
				(F^*_j + F_{7-j}z)(F^*_i + F_{7-i}z)
			\end{aligned}
		\end{equation}
		for all $z \in \mathbb{T}$ and $1 \leqslant i, j \leqslant 6$. It is immediate from \eqref{C 1.3} that
		\begin{equation}\label{C 1.4}
			\begin{aligned}
				(F^*_i + F_{7-i}\omega)(F^*_j + F_{7-j}\omega) &=
				(F^*_j + F_{7-j}\omega)(F^*_i + F_{7-i}\omega)
			\end{aligned}
		\end{equation}
		for all $\omega \in \mathbb{T}$ and $1 \leqslant i, j \leqslant 6$. Thus it follows from \eqref{C 1.4} that
		\begin{equation*}
			\begin{aligned}
				\left(
				\begin{pmatrix}
					M_{F^*_1 + F_6\omega} & 0\\
					0 & N^{(1)}_D
				\end{pmatrix}, \dots,
				\begin{pmatrix}
					M_{F^*_6 + F_1\omega} & 0\\
					0 & N^{(6)}_D
				\end{pmatrix},
				\begin{pmatrix}
					M_\omega & 0\\
					0 & N^{(7)}_D
				\end{pmatrix}
				\right)
			\end{aligned}
		\end{equation*}
		is commutative. Furthermore, the extensions $M_{F^*_i + F_{7-i}\omega}$ acting on $L^2(\mathcal{E})$ of the operators $M_{F^*_i + F_{7-i}z}$ on $H^2(\mathcal{E})$ is norm-preserving. Thus whenever the operator norm of $M_{F^*_i + F_{7-i}z}$ does not exceed one, then the operator norm of $M_{F^*_i + F_{7-i}\omega}$ does not exceed one. On the other hand, it can also be deduced that
		\begin{equation*}
			\begin{aligned}
				M_{F^*_i + F_{7-i}\omega}
				&= M^*_{F^*_{7-i} + F_i\omega}M_{\omega}
			\end{aligned}
		\end{equation*}
		for all $\omega \in \mathbb{T}$ and $1 \leqslant i \leqslant 6$. Therefore, by [Theorem $4.4$, \cite{apal2}],
		\begin{equation}
			\begin{aligned}
				\left(
				\begin{pmatrix}
					M_{F^*_1 + F_6\omega} & 0\\
					0 & N^{(1)}_D
				\end{pmatrix}, \dots,
				\begin{pmatrix}
					M_{F^*_6 + F_1\omega} & 0\\
					0 & N^{(6)}_D
				\end{pmatrix},
				\begin{pmatrix}
					M_\omega & 0\\
					0 & N^{(7)}_D
				\end{pmatrix}
				\right)
			\end{aligned}
		\end{equation}
		is the required
		$\Gamma_{E(3; 3; 1, 1, 1)}$-unitary on
		$
		\begin{pmatrix}
			L^2(\mathcal{E})\\
			\mathcal{F}
		\end{pmatrix}$. This completes the proof.
	\end{proof}
	
	We only state the following result, as its proof is exactly similar to that of Corollary \ref{Cor 1}.
	
	\begin{cor}\label{Cor 2}
		Let $\textbf{S} = (S_1, S_2, S_3, \tilde{S}_1, \tilde{S}_2)$ be a commuting $5$-tuple of bounded operators on a Hilbert space $\mathcal{H}$. Then $\textbf{S}$ is a $\Gamma_{E(3; 2; 1, 2)}$-isometry if and only if it can be extended to a $\Gamma_{E(3; 2; 1, 2)}$-unitary acting on a Hilbert space of minimal unitary dilation of the isometry $S_3$.
	\end{cor}
	
	\section{Douglas Type Functional Model for $\Gamma_{E(3; 3; 1, 1, 1)}$-Contraction and $\Gamma_{E(3; 2; 1, 2)}$-Contraction}\label{Douglas Type Functional Model}
	
	The classical \textit{Douglas model} for a contraction $T$ acting on a Hilbert space $\mathcal{H}$ can be found in \cite{Douglas}. In this section, we develop Douglas type functional model for $\Gamma_{E(3; 3; 1, 1, 1)}$-contraction and $\Gamma_{E(3; 2; 1, 2)}$-contraction.
	
	Let $\textbf{T} = (T_1, \dots, T_7)$ be a $\Gamma_{E(3; 3; 1, 1, 1)}$-contraction on Hilbert space $\mathcal{H}$ with $F_1, \dots, F_6$ be the fundamental operators for $\textbf{T}^* = (T^*_1, \dots, T^*_7)$. Let $\textbf{N} = (N^{(1)}_D, \dots, N^{(6)}_D, N^{(7)}_D)$ defined on $\mathcal{Q}_{T^*_7}$ be the canonical $\Gamma_{E(3; 3; 1, 1, 1)}$-unitary associated with $\textbf{T}$. Define $\Pi^{\textbf{T}}_D =
	\begin{pmatrix}
		\mathcal{O}_{D_{T^*_7}, T^*_7}\\
		Q_{T^*_7}
	\end{pmatrix}$ where $\mathcal{O}_{D_{T^*_7}, T^*_7}(z) = \sum_{n \geqslant 0} z^nD_{T^*_7}T^{*n}_7$. Notice that
		\begin{equation}\label{Pi^T_D isometry}
			\begin{aligned}
				||\Pi^{\textbf{T}}_D h||^2
				&= ||\mathcal{O}_{D_{T^*_7}, T^*_7}(z)h||^2_{H^2(\mathcal{D}_{T^*_7})} + ||Q_{T^*_7}h||^2\\
				&= \sum_{n \geqslant 0} ||D_{T^*_7}T^*_7h||^2 + \lim_{n \to \infty} ||T^*_7h||^2\\
				&= \lim_{n \to \infty} \sum_{k = 1}^{n} (||T^{*k-1}_7h||^2 - ||T^{*k}_7h||^2) + \lim_{n \to \infty} ||T^*_7h||^2\\
				&= ||h||.
			\end{aligned}
		\end{equation}
		This implies that $\Pi^{\textbf{T}}_D$ is an isometry and hence $\Pi^{\textbf{T}}_D$ is an isometric embedding of $\mathcal{H}$ into
		$\begin{pmatrix}
			H^2(\mathcal{D}_{T^*_7})\\
			\mathcal{Q}_{T^*}
		\end{pmatrix}$.
	
	We show that
	\begin{equation}
		\begin{aligned}
			P_{\mathcal{H}^{\textbf{T}}_D}
			\left(
			\begin{pmatrix}
				M_{F^*_1 + F_6z} & 0\\
				0 & N^{(1)}_D
			\end{pmatrix},
			\dots,
			\begin{pmatrix}
				M_{F^*_6 + F_1z} & 0\\
				0 & N^{(6)}_D
			\end{pmatrix},
			\begin{pmatrix}
				M_z & 0\\
				0 & N^{(7)}_D
			\end{pmatrix}
			\right)\Bigg|_{\mathcal{H}^{\textbf{T}}_D}
		\end{aligned}
	\end{equation}
	is a functional model for $\textbf{T}$ with $\mathcal{H}^{\textbf{T}}_D
	:= \Ran \Pi^{\textbf{T}}_D \subset
	\begin{pmatrix}
		H^2(\mathcal{D}_{T^*_7)})\\
		\mathcal{Q}_{T^*_7}
	\end{pmatrix}$ is the corresponding model space, where $\tilde{F}_1, \dots, \tilde{F}_6$ are the fundamental operators of $\textbf{T}^*$.
	
	\begin{thm}[Douglas Model for $\Gamma_{E(3; 3; 1, 1, 1)}$-Contraction]\label{Douglas Model 1}
		Let $\textbf{T} = (T_1, \dots, T_7)$ be a $\Gamma_{E(3; 3; 1, 1, 1)}$-contraction on a Hilbert space $\mathcal{H}$ with $\tilde{F}_1, \dots, \tilde{F}_6$ be the fundamental operators of $\textbf{T}^*$. Suppose $(N^{(1)*}_D, \dots, N^{(6)*}_D, N^{(7)*}_D)$ be the canonical $\Gamma_{E(3; 3; 1, 1, 1)}$-unitary associated to $\textbf{T}$ and $\Pi^{\textbf{T}}_D =
		\begin{pmatrix}
			\mathcal{O}_{D_{T^*_7}, T^*_7}\\
			Q_{T^*_7}
		\end{pmatrix}$ be the Douglas isometric embedding of $\mathcal{H}$ into
		$
		\begin{pmatrix}
			H^2(\mathcal{D}_{T^*_7})\\
			\mathcal{Q}_{T^*_7}
		\end{pmatrix}$. Then $\textbf{T}$ is unitarily equivalent to
		\begin{equation}\label{Douglas Model 1.1}
			\begin{aligned}
				P_{\mathcal{H}^{\textbf{T}}_D}
				\left(
				\begin{pmatrix}
					M_{\tilde{F}^*_1 + \tilde{F}_6z} & 0\\
					0 & N^{(1)}_D
				\end{pmatrix},
				\dots,
				\begin{pmatrix}
					M_{\tilde{F}^*_6 + \tilde{F}_1z} & 0\\
					0 & N^{(6)}_D
				\end{pmatrix},
				\begin{pmatrix}
					M_z & 0\\
					0 & N^{(7)}_D
				\end{pmatrix}
				\right)\Bigg|_{\mathcal{H}^{\textbf{T}}_D}
			\end{aligned}
		\end{equation}
		($\tilde{F}_1, \dots, \tilde{F}_6$ satisfy \eqref{Commutative A} when \eqref{Douglas Model 1.1} is commutative) where $\mathcal{H}^{\textbf{T}}_D$ is the functional model space of $\textbf{T}$ given by
		\begin{equation}\label{Model Space 1}
			\begin{aligned}
				\mathcal{H}^{\textbf{T}}_D
				&:= \Ran \Pi^{\textbf{T}}_D \subset
				\begin{pmatrix}
					H^2(\mathcal{D}_{T^*_7})\\
					\mathcal{Q}_{T^*_7}
				\end{pmatrix}.
			\end{aligned}
		\end{equation}
	\end{thm}
	
	\begin{proof}
		Notice that $\Pi^{\textbf{T}}_D : \mathcal{H} \to \Ran \Pi^{\textbf{T}}_D$ is unitary. In order to prove $\textbf{T}$ is unitarily equivalent to \eqref{Douglas Model 1.1}, it is enough to establish
		\begin{equation}\label{DM 1.2}
			\begin{aligned}
				\Pi^{\textbf{T}}_D(T^*_1, \dots, T^*_7)
				&=
				\left(
				\begin{pmatrix}
					M_{\tilde{F}^*_1 + \tilde{F}_6z} & 0\\
					0 & N^{(1)}_D
				\end{pmatrix}^*,
				\dots,
				\begin{pmatrix}
					M_{\tilde{F}^*_6 + \tilde{F}_1z} & 0\\
					0 & N^{(6)}_D
				\end{pmatrix}^*,
				\begin{pmatrix}
					M_z & 0\\
					0 & N^{(7)}_D
				\end{pmatrix}^*
				\right)\Pi^{\textbf{T}}_D,
			\end{aligned}
		\end{equation}
		which is equivalent to the following:
		\begin{equation}\label{DM 1.3}
			\begin{aligned}
				&\mathcal{O}_{D_{T^*_7}, T^*_7}(z)(T^*_1, \dots, T^*_6, T^*_7) = (M^*_{\tilde{F}^*_1 + \tilde{F}_6z}, \dots, M^*_{\tilde{F}^*_6 + \tilde{F}_1z}, M^*_z)\mathcal{O}_{D_{T^*_7}, T^*_7}(z),
			\end{aligned}
		\end{equation}
		$$\rm{and}$$
		\begin{equation}\label{DM 1.4}
			\begin{aligned}
				&Q_{T^*_7}(T^*_1, \dots, T^*_6, T^*_7) = (N^{(1)*}_D, \dots, N^{(6)*}_D, N^{(7)*}_D)Q_{T^*_7}.
			\end{aligned}
		\end{equation}
		By \eqref{V^i*_T^*_7} of the canonical construction of $\Gamma_{E(3; 3; 1, 1, 1)}$-unitary it is immediate that \eqref{DM 1.4} holds. Thus, we only show \eqref{DM 1.3}. Since, $(T^*_1, \dots, T^*_7)$ is a $\Gamma_{E(3; 3; 1, 1, 1)}$-contraction then by applying Proposition \ref{FiFj} to $\textbf{T}^*$ we have
		\begin{equation}\label{DM 1.5}
			\begin{aligned}
				&D_{T^*_7}T^*_i = \tilde{F}_iD_{T^*_7} + \tilde{F}^*_{7-i}D_{T^*_7}T^*_7 \,\, \text{for} \,\, 1 \leqslant i \leqslant 6,
			\end{aligned}
		\end{equation}
		where $\tilde{F}_1, \dots, \tilde{F}_6$ defined on $\mathcal{D}_{T^*_7}$ are the fundamental operators of $\textbf{T}^*$. Since, $T^*_7$ is contraction then $(I - zT^*_7)$ is invertible and thus multiplying $(I - zT^*_7)^{-1}$ both side of \eqref{DM 1.5} we have
		\begin{equation}\label{DM 1.6}
			\begin{aligned}
				&D_{T^*_7}T^*_i(I - zT^*_7)^{-1} = (\tilde{F}_iD_{T^*_7} + \tilde{F}^*_{7-i}D_{T^*_7}T^*_7)(I - zT^*_7)^{-1} \,\, \text{for} \,\, 1 \leqslant i \leqslant 6.
			\end{aligned}
		\end{equation}
		As $T_i$ commutes with $T_7$ then from \eqref{DM 1.6} we get
		\begin{equation}\label{DM 1.7}
			\begin{aligned}
				&D_{T^*_7}(I - zT^*_7)^{-1}T^*_i = \tilde{F}_iD_{T^*_7}(I - zT^*_7)^{-1} + \tilde{F}^*_{7-i}D_{T^*_7}(I - zT^*_7)^{-1}T^*_7.
			\end{aligned}
		\end{equation}
		By routine computation one can obtain that \eqref{DM 1.7} is equivalent to \eqref{DM 1.3}.	This completes the proof.
	\end{proof}
	
	It is important to observe that
	\begin{equation*}
		\begin{aligned}
			\left(
			\begin{pmatrix}
				M_{\tilde{F}^*_1 + \tilde{F}_6z} & 0\\
				0 & N^{(1)}_D
			\end{pmatrix},
			\dots,
			\begin{pmatrix}
				M_{\tilde{F}^*_6 + \tilde{F}_1z} & 0\\
				0 & N^{(6)}_D
			\end{pmatrix},
			\begin{pmatrix}
				M_z & 0\\
				0 & N^{(7)}_D
			\end{pmatrix}
			\right)
		\end{aligned}
	\end{equation*}
	with $\tilde{F}_1, \dots, \tilde{F}_6$ are the fundamental operators of $\textbf{T}^*$, is a $\Gamma_{E(3; 3; 1, 1, 1)}$-isometry. In the next theorem we prove that if a $\textbf{T}$ is a $\Gamma_{E(3; 3; 1, 1, 1)}$-contraction with $T_7$ has a minimal isometric dilation $V_7$ on a larger Hilbert space $\mathcal{K}$ containing $\mathcal{H}$ then $\textbf{T}$ has unique isometric lift to $\mathcal{K}$.
	
	\begin{thm}\label{Existence and Uniqueness 1}
		Let $\textbf{T} = (T_1, \dots, T_7)$ be a $\Gamma_{E(3; 3; 1, 1, 1)}$-contraction on a Hilbert space $\mathcal{H}$ with $\tilde{F}_1, \dots, \tilde{F}_6$ be the fundamental operators for $\textbf{T}^*$ and $V_7$ be the minimal isometric dilation of $T_7$ on a larger Hilbert space $\mathcal{K}$ containing $\mathcal{H}$. Then there exists unique operators $V_1, \dots, V_6$ acting on $\mathcal{K}$ such that $\textbf{V} = (V_1, \dots, V_7)$ is a $\Gamma_{E(3; 3; 1, 1, 1)}$-isometric dilation of $\textbf{T}$ provided $F_1, \dots, F_6$ satisfies \eqref{Commutative A}.
	\end{thm}
	
	\begin{proof}
		It is clear form Theorem \ref{Douglas Model 1} that
		\begin{equation*}
			\begin{aligned}
				\left(V_1 =
				\begin{pmatrix}
					M_{\tilde{F}^*_1 + \tilde{F}_6z} & 0\\
					0 & N^{(1)}_D
				\end{pmatrix},
				\dots,
				V_6 =
				\begin{pmatrix}
					M_{\tilde{F}^*_6 + \tilde{F}_1z} & 0\\
					0 & N^{(6)}_D
				\end{pmatrix},
				V_7 =
				\begin{pmatrix}
					M_z & 0\\
					0 & N^{(7)}_D
				\end{pmatrix}
				\right)
			\end{aligned}
		\end{equation*}
		is a $\Gamma_{E(3; 3; 1, 1, 1)}$-isometry. This proves the existence of $\Gamma_{E(3; 3; 1, 1, 1)}$-isometric dilation of $\textbf{T}$.
		
		To prove the uniqueness, let $\Pi^{\textbf{T}}_D =
		\begin{pmatrix}
			\mathcal{O}_{D_{T^*_7}, T^*_7}\\
			Q_{T^*_7}
		\end{pmatrix}$ be the Douglas isometric embedding and
		$
		\begin{pmatrix}
			H^2(\mathcal{D}_{T^*_7})\\
			\mathcal{Q}_{T^*_7}
		\end{pmatrix}$ be the model space for minimal isometric lift of $T_7$ with $V^{(7)}_D =
		\begin{pmatrix}
			M_z & 0\\
			0 & N^{(7)}_D
		\end{pmatrix}$ be the minimal isometric dilation of $T_7$. Suppose $V_i$'s are unitarily equivalent to
		\begin{equation}\label{Uniq 1}
			\begin{aligned}
				V^{(i)}_D &=
				\begin{pmatrix}
					V^{(i)}_{11} & V^{(i)}_{12}\\
					V^{(i)}_{21} & V^{(i)}_{22}
				\end{pmatrix} \,\, \text{for} \,\, 1 \leqslant i \leqslant 6.
			\end{aligned}
		\end{equation}
		with respect to the embedding $\Pi_D$. Thus,
		\begin{equation}\label{Uniq 2}
			\begin{aligned}
				\textbf{V}_D&=
				\left(
				\begin{pmatrix}
					V^{(1)}_{11} & V^{(1)}_{12}\\
					V^{(1)}_{21} & V^{(1)}_{22}
				\end{pmatrix}, \dots,
				\begin{pmatrix}
					V^{(6)}_{11} & V^{(6)}_{12}\\
					V^{(6)}_{21} & V^{(6)}_{22}
				\end{pmatrix},
				\begin{pmatrix}
					M_z & 0\\
					0 & N^{(7)}_D
				\end{pmatrix}
				\right)
			\end{aligned}
		\end{equation}
		is a $\Gamma_{E(3; 3; 1, 1, 1)}$-isometry. Then by proceeding similarly as \eqref{M 1.1} and \eqref{M 1.2} we obtain
		\begin{equation}\label{Uniq 3}
			\begin{aligned}
				V^{(i)}_D &=
				\begin{pmatrix}
					M_{X^*_i + X_{7-i}z} & 0\\
					0 & V^{(i)}_{22}
				\end{pmatrix} \,\, \text{for} \,\, 1 \leqslant i \leqslant 6
			\end{aligned}
		\end{equation}
		for some $X_i \in \mathcal{B}(\mathcal{D}_{T^*_7})$ such that $X^*_i + X_{7-i}z$ is a contraction for all $z \in \mathbb{D}$. Since $\textbf{V}_D$ is a $\Gamma_{E(3; 3; 1, 1, 1)}$-isometric lift of $\textbf{T}$ then $V^{(i)}_D$'s satisfy the following:
		\begin{equation}\label{Uniq 4}
			\begin{aligned}
				\begin{pmatrix}
					M^*_{X^*_i + X_{7-i}z} & 0\\
					0 & V^{(i)*}_{22}
				\end{pmatrix}
				\begin{pmatrix}
					\mathcal{O}_{D_{T^*_7}, T^*_7}\\
					Q_{T^*_7}
				\end{pmatrix} &=
				\begin{pmatrix}
					\mathcal{O}_{D_{T^*_7}, T^*_7}\\
					Q_{T^*_7}
				\end{pmatrix}T^*_i \,\, \text{for} \,\, 1 \leqslant i \leqslant 6.
			\end{aligned}
		\end{equation}
		Note that \eqref{Uniq 4} can be split into
		\begin{equation}\label{Uniq 5}
			\begin{aligned}
				M^*_{X^*_i + X_{7-i}z}\mathcal{O}_{D_{T^*_7}, T^*_7}
				&= \mathcal{O}_{D_{T^*_7}, T^*_7}T^*_i
			\end{aligned}
		\end{equation}
		$$\rm{and}$$
		\begin{equation}\label{Uniq 6}
			\begin{aligned}
				V^{(i)*}_{22}Q_{T^*_7}
				&= Q_{T^*_7}T^*_i
			\end{aligned}
		\end{equation}
		for $1 \leqslant i \leqslant 6$. We show that $V^{(i)}_{22} = N^{(i)}_D$ and $X_i = F_i$ for $1 \leqslant i
		\leqslant 6$. Observe that $V^{(i)}_{22}N^{(7)}_D = N^{(7)}_DV^{(i)}_{22}$ and as $N^{(7)}_D$ is unitary then from \eqref{Uniq 6} we get
		\begin{equation}
			\begin{aligned}
				V^{(i)*}_{22}(N^{(7)n}_DQ_{T^*_7}h)
				&= N^{(7)n}_DV^{(i)*}_{22}Q_{T^*_7}h\\
				&= N^{(7)n}_DQ_{T^*_7}T^*_ih\\
				&= N^{(7)n}_DN^{(i)*}_DQ_{T^*_7}T^*_ih\\
				&= N^{(i)*}_D(N^{(7)n}_DQ_{T^*_7}T^*_ih).
			\end{aligned}
		\end{equation}
		Since, $\{N^{(7)n}_DQ_{T^*_7}h : n \geqslant 0, h \in \mathcal{H}\}$ is dense in $\mathcal{Q}_{T^*_7}$ then it follows that $V^{(i)*}_{22} = N^{(i)*}_D$ for $1 \leqslant i \leqslant 6$.
		
		Observe that \eqref{Uniq 5} can be written by expanding the power series of $\mathcal{O}_{D_{T^*_7}, T^*_7}$ as
		\begin{equation}
			\begin{aligned}
				X_iD_{T^*_7} + X^*_{7-i}D_{T^*_7}T^*_7
				&= D_{T^*_7}T^*_i \,\, \text{for} \,\, 1 \leqslant i \leqslant 6.
			\end{aligned}
		\end{equation}
		By [Theorem $2.7$, \cite{apal3}] it follows that $X_1, \dots, X_6$ are the fundamental operators of $\textbf{T}^*$; that is $X_i = \tilde{F}_i$ for $1 \leqslant i \leqslant 6$. This proves the uniqueness of $V_1, \dots, V_6$. Hence, the proof is done.
	\end{proof}
	
	Let $\textbf{S} = (S_1, S_2, S_3, \tilde{S}_1, \tilde{S}_2)$ be a $\Gamma_{E(3; 2; 1, 2)}$-contraction on a Hilbert space $\mathcal{H}$ with $\hat{G}_1, 2\hat{G}_2, 2\hat{\tilde{G}}_1, \hat{\tilde{G}}_2$ be the fundamental operators of $\textbf{S}^*$. Suppose $\textbf{M} = (M^{(1)}_D, M^{(2)}_D, M^{(3)}_D, \tilde{M}^{(1)}_D, \tilde{M}^{(2)}_D)$ defined on $\mathcal{Q}_{S^*_3}$ be the canonical $\Gamma_{E(3; 2; 1, 2)}$-unitary associated to $\textbf{S}$. We define $\Pi^{\textbf{S}}_D =
	\begin{pmatrix}
		\mathcal{O}_{D_{S^*_3}, S^*_3}\\
		Q_{S^*_3}
	\end{pmatrix}$ where $\mathcal{O}_{D_{S^*_3}, S^*_3}(z) = \sum_{n \geqslant 0} z^nD_{S^*_3}S^{*n}_3$. It follows similarly as \eqref{Pi^T_D isometry} that $\Pi^{\textbf{S}}_D$ is an isometry. We prove that
	\begin{equation}
		\begin{aligned}
			&P_{\mathcal{H}^{\textbf{S}}_D}
			\left(
			\begin{pmatrix}
				M_{\hat{G}^*_1 + \hat{\tilde{G}}_2z} & 0\\
				0 & M^{(1)}_D
			\end{pmatrix},
			\begin{pmatrix}
				M_{2\hat{G}^*_2 + 2\hat{\tilde{G}}_1z} & 0\\
				0 & M^{(2)}_D
			\end{pmatrix},
			\begin{pmatrix}
				M_z & 0\\
				0 & M^{(3)}_D
			\end{pmatrix}, \right.\\
			&\left. \hspace{2.5cm}
			\begin{pmatrix}
				M_{2\hat{\tilde{G}}^*_1 + 2\hat{G}_2z} & 0\\
				0 & \tilde{M}^{(1)}_D
			\end{pmatrix}
			\begin{pmatrix}
				M_{\hat{\tilde{G}}^*_2 + \hat{G}_1z} & 0\\
				0 & \tilde{M}^{(2)}_D
			\end{pmatrix}
			\right)\Bigg|_{\mathcal{H}^{\textbf{S}}_D}
		\end{aligned}
	\end{equation}
	is a model for $\textbf{S}$ with $\mathcal{H}^{\textbf{S}}_D
	:= \Ran \Pi^{\textbf{S}}_D \subset
	\begin{pmatrix}
		H^2(\mathcal{D}_{S^*_3)})\\
		\mathcal{Q}_{S^*_3}
	\end{pmatrix}$ is the functional model space, where $\hat{G}_1, 2\hat{G}_2, 2\hat{\tilde{G}}_1, \hat{\tilde{G}}_2$ are the fundamental operators of $\textbf{S}^*$.
	
	We only state the following theorem as the proof is similar to that of Theorem \ref{Douglas Model 1}.
	
	\begin{thm}[Douglas Model for $\Gamma_{E(3; 2; 1, 2)}$-Contraction]\label{Douglas Model 2}
		Let $\textbf{S} = (S_1, S_2, S_3, \tilde{S}_1, \tilde{S}_2)$ be a $\Gamma_{E(3; 2; 1, 2)}$-contraction on a Hilbert space $\mathcal{H}$ with $\hat{G}_1, 2\hat{G}_2, 2\hat{\tilde{G}}_1, \hat{\tilde{G}}_2$ be the fundamental operators of $\textbf{S}^*$. Suppose $(M^{(1)*}_D, M^{(2)*}_D, M^{(3)*}_D,\\ \tilde{M}^{(1)*}_D, \tilde{M}^{(2)*}_D)$ is the canonical $\Gamma_{E(3; 2; 1, 2)}$-unitary associated to $\textbf{S}$ and $\Pi^{\textbf{S}}_D =
		\begin{pmatrix}
			\mathcal{O}_{D_{S^*_3}, S^*_3}\\
			Q_{S^*_3}
		\end{pmatrix}$ be the Douglas isometric embedding of $\mathcal{H}$ into
		$
		\begin{pmatrix}
			H^2(\mathcal{D}_{S^*_3})\\
			\mathcal{Q}_{S^*_3}
		\end{pmatrix}$. Then $\textbf{S}$ is unitarily equivalent to
		\begin{equation}\label{Douglas Model 2.1}
			\begin{aligned}
				&P_{\mathcal{H}^{\textbf{S}}_D}
				\left(
				\begin{pmatrix}
					M_{\hat{G}^*_1 + \hat{\tilde{G}}_2z} & 0\\
					0 & M^{(1)}_D
				\end{pmatrix},
				\begin{pmatrix}
					M_{2\hat{G}^*_2 + 2\hat{\tilde{G}}_1z} & 0\\
					0 & M^{(2)}_D
				\end{pmatrix},
				\begin{pmatrix}
					M_z & 0\\
					0 & M^{(3)}_D
				\end{pmatrix}, \right.\\
				&\left. \hspace{2cm}
				\begin{pmatrix}
					M_{2\hat{\tilde{G}}^*_1 + 2\hat{G}_2z} & 0\\
					0 & \tilde{M}^{(1)}_D
				\end{pmatrix}
				\begin{pmatrix}
					M_{\hat{\tilde{G}}^*_2 + \hat{G}_1z} & 0\\
					0 & \tilde{M}^{(2)}_D
				\end{pmatrix}
				\right)\Bigg|_{\mathcal{H}^{\textbf{S}}_D}
			\end{aligned}
		\end{equation}
		($\hat{G}_1, 2\hat{G}_2, 2\hat{\tilde{G}}_1, \hat{\tilde{G}}_2$ satisfy \eqref{Commutative B(i)} and \eqref{Commutative B(ii)} when \eqref{Douglas Model 2.1} is commutative) where $\mathcal{H}^{\textbf{S}}_D$ is the functional model space of $\textbf{S}$ given by
		\begin{equation}\label{Model Space 2}
			\begin{aligned}
				\mathcal{H}^{\textbf{S}}_D
				&:= \Ran \Pi^{\textbf{S}}_D \subset
				\begin{pmatrix}
					H^2(\mathcal{D}_{S^*_3})\\
					\mathcal{Q}_{S^*_3}
				\end{pmatrix}.
			\end{aligned}
		\end{equation}
	\end{thm}

	From Theorem \ref{Douglas Model 2}, we notice that
	\begin{equation*}
		\begin{aligned}
			&\left(
			\begin{pmatrix}
				M_{\hat{G}^*_1 + \hat{\tilde{G}}_2z} & 0\\
				0 & M^{(1)}_D
			\end{pmatrix},
			\begin{pmatrix}
				M_{2\hat{G}^*_2 + 2\hat{\tilde{G}}_1z} & 0\\
				0 & M^{(2)}_D
			\end{pmatrix},
			\begin{pmatrix}
				M_z & 0\\
				0 & M^{(3)}_D
			\end{pmatrix}, \right.\\
			&\left. \hspace{2cm}
			\begin{pmatrix}
				M_{2\hat{\tilde{G}}^*_1 + 2\hat{G}_2z} & 0\\
				0 & \tilde{M}^{(1)}_D
			\end{pmatrix}
			\begin{pmatrix}
				M_{\hat{\tilde{G}}^*_2 + \hat{G}_1z} & 0\\
				0 & \tilde{M}^{(2)}_D
			\end{pmatrix}
			\right)
		\end{aligned}
	\end{equation*}
	with $\hat{G}_1, 2\hat{G}_2, 2\hat{\tilde{G}}_1, \hat{\tilde{G}}_2$ are the fundamental operators of $\textbf{S}^*$, is a $\Gamma_{E(3; 2; 1, 2)}$-isometry. We provide that if $\textbf{S}$ is a $\Gamma_{E(3; 2; 1, 2)}$-contraction with $W_3$ is the minimal isometric dilation $S_3$ on a larger Hilbert space $\mathcal{K}$ containing $\mathcal{H}$ then $\textbf{S}$ has unique isometric lift to $\mathcal{K}$. We omit the proof of the following theorem as it is similar to proof of Theorem \ref{Existence and Uniqueness 1}.
	
	\begin{thm}\label{Existence and Uniqueness 2}
		Let $\textbf{S} = (S_1, S_2, S_3, \tilde{S}_1, \tilde{S}_2)$ be a $\Gamma_{E(3; 2; 1, 2)}$-contraction on a Hilbert space $\mathcal{H}$ and $W_3$ be the minimal isometric dilation of $S_3$ on a larger Hilbert space $\mathcal{K}$ containing $\mathcal{H}$. Then there exists unique operators $W_1, W_2, \tilde{W}_1, \tilde{W}_2$ acting on $\mathcal{K}$ such that $\textbf{W} = (W_1, W_2, W_3, \tilde{W}_1, \tilde{W}_2)$ is a $\Gamma_{E(3; 2; 1, 2)}$-isometric dilation of $\textbf{S}$.
	\end{thm}

	\section{Sz.-Nagy-Foias Type Functional Model for C.N.U. $\Gamma_{E(3; 3; 1, 1, 1)}$-Contraction and C.N.U. $\Gamma_{E(3; 2; 1, 2)}$-Contraction}\label{Nagy-Foias Type Functional Model}
	
	The classical Nagy-Foias model for c.n.u. contraction can be found \cite{Nagy}. We demonstrate Nagy-Foias type functional model for c.n.u. $\Gamma_{E(3; 3; 1, 1, 1)}$-contraction and c.n.u. $\Gamma_{E(3; 2; 1, 2)}$-contraction in this section.

	Let $\textbf{T} = (T_1, \dots, T_7)$ be a c.n.u. $\Gamma_{E(3; 3; 1, 1, 1)}$-contraction on a Hilbert space $\mathcal{H}$. Thus $T_7$ is c.n.u. contraction. Consider the function
	\begin{equation}\label{Delta_{T_7}}
		\begin{aligned}
			\triangle_{T_7}(\omega) &=
			(I - \Theta_{T_7}(\omega)^*\Theta_{T_7}(\omega))^{1/2}.
		\end{aligned}
	\end{equation}
	According to Sz.-Nagy and Foias,
	\begin{equation}\label{V^T_NF}
		\begin{aligned}
			V^{\textbf{T}}_{NF} &:=
			\begin{pmatrix}
				M_z & 0\\
				0 & M_{\omega}|_{\overline{\triangle_{T_7} L^2(\mathcal{D}_{T_7})}}
			\end{pmatrix} :
			\begin{pmatrix}
				H^2(\mathcal{D}_{T^*_7})\\
				\overline{\triangle_{T_7} L^2(\mathcal{D}_{T_7})}
			\end{pmatrix} \to
			\begin{pmatrix}
				H^2(\mathcal{D}_{T^*_7})\\
				\overline{\triangle_{T_7} L^2(\mathcal{D}_{T_7})}
			\end{pmatrix}
		\end{aligned}
	\end{equation}
	is a minimal isometric dilation of $T_7$ via the corresponding isometric embedding
	\begin{equation}\label{Pi^T_NF}
		\begin{aligned}
			\Pi^{\textbf{T}}_{NF} : \mathcal{H} \to
			\begin{pmatrix}
				H^2(\mathcal{D}_{T^*_7})\\
				\overline{\triangle_{T_7} L^2(\mathcal{D}_{T_7})}
			\end{pmatrix}
		\end{aligned}
	\end{equation}
	such that
	\begin{equation}\label{H^T_NF}
		\begin{aligned}
			\mathcal{H}^{\textbf{T}}_{NF} &= \Ran \Pi^{\textbf{T}}_{NF} =
			\begin{pmatrix}
				H^2(\mathcal{D}_{T^*_7})\\
				\overline{\triangle_{T_7} L^2(\mathcal{D}_{T_7})}
			\end{pmatrix} \ominus
			\begin{pmatrix}
				\Theta_{T_7}\\
				\triangle_{T_7}
			\end{pmatrix}
			H^2(\mathcal{D}_{T_7}).
		\end{aligned}
	\end{equation}
	
	Since, Sz.-Nagy-Foias iaometric dilation and Douglas isometric dilations are minimal then they are unitarily equivalent. In another word, there exists a unitary
	$\Phi :
	\begin{pmatrix}
		H^2(\mathcal{D}_{T^*_7})\\
		\mathcal{Q}_{T^*_7}
	\end{pmatrix}$
	such that
	\begin{equation}\label{Pi^T_NF, Pi^T_D relation}
		\begin{aligned}
			\Pi^{\textbf{T}}_{NF} &=
			\Phi
			\Pi^{\textbf{T}}_D.
		\end{aligned}
	\end{equation}
	In \cite{J. A. Ball}, Ball and Sau showed that there exists a unitary $u^{\textbf{T}}_{\min} : \mathcal{Q}_{T^*_7} \to \overline{\triangle_{T_7} L^2(\mathcal{D}_{T_7})}$ such that
	\begin{equation}\label{u^T_min, N^7_D, M_w}
		\begin{aligned}
			u^{\textbf{T}}_{\min}N^{(7)}_D &= M_{\omega}|_{\overline{\triangle_{T_7} L^2(\mathcal{D}_{T_7})}}u^{\textbf{T}}_{\min}
		\end{aligned}
	\end{equation}
	$$\rm{and}$$
	\begin{equation}\label{Pi^T_NF, Pi^T_D}
		\begin{aligned}
			\Pi^{\textbf{T}}_{NF} &=
			\begin{pmatrix}
				I_{H^2(\mathcal{D}_{T^*_7})} & 0\\
				0 & u^{\textbf{T}}_{\min}
			\end{pmatrix}
			\Pi^{\textbf{T}}_D.
		\end{aligned}
	\end{equation}	
	We define $N^{(i)}_{NF}$ on $\overline{\triangle_{T_7} L^2(\mathcal{D}_{T_7})}$ by
	\begin{equation}\label{N^i_NF}
		\begin{aligned}
			N^{(i)}_{NF} &= u^{\textbf{T}}_{\min}N^{(i)}_Du^{\textbf{T}*}_{\min} \,\, \text{for} \,\, 1 \leqslant i \leqslant 7.
		\end{aligned}
	\end{equation}
	Therefore, the following theorem on functional model for c.n.u. $\Gamma_{E(3; 3; 1, 1, 1)}$-contraction is straightforward application of \eqref{Pi^T_NF, Pi^T_D} and Theorem \ref{Douglas Model 1}.
	
	\begin{thm}[Sz.-Nagy-Foias Model for C.N.U. $\Gamma_{E(3; 3; 1, 1, 1)}$-Contraction]\label{Nagy-Foias Model 1}
		Let $\textbf{T} = (T_1, \dots, T_7)$ be a c.n.u. $\Gamma_{E(3; 3; 1, 1, 1)}$-contraction on a Hilbert space $\mathcal{H}$ with $F_1, \dots, F_6$ be the fundamental operators of $\textbf{T}^*$. Then $\textbf{T}$ is unitarily equivalent to
		\begin{equation}
			\begin{aligned}
				P_{\mathcal{H}^{\textbf{T}}_{NF}}
				\left(
				\begin{pmatrix}
					M_{\tilde{F}^*_1 + \tilde{F}_6z} & 0\\
					0 & N^{(1)}_{NF}
				\end{pmatrix},
				\dots,
				\begin{pmatrix}
					M_{\tilde{F}^*_6 + \tilde{F}_1z} & 0\\
					0 & N^{(6)}_{NF}
				\end{pmatrix},
				\begin{pmatrix}
					M_z & 0\\
					0 & N^{(7)}_{NF}
				\end{pmatrix}
				\right)\Bigg|_{\mathcal{H}^{\textbf{T}}_{NF}}
			\end{aligned}
		\end{equation}
		($\tilde{F}_1, \dots, \tilde{F}_6$ satisfy \eqref{Commutative A} when \eqref{Douglas Model 1.1} is commutative) where $\mathcal{H}^{\textbf{T}}_{NF}$ is the functional model space of $\textbf{T}$ defined by
		\begin{equation}
			\begin{aligned}
				\mathcal{H}^{\textbf{T}}_{NF} &:=
				\Ran \Pi^{\textbf{T}}_{NF} =
				\begin{pmatrix}
					H^2(\mathcal{D}_{T^*_7})\\
					\overline{\triangle_{T_7} L^2(\mathcal{D}_{T_7})}
				\end{pmatrix} \ominus
				\begin{pmatrix}
					\Theta_{T_7}\\
					\triangle_{T_7}
				\end{pmatrix}
				H^2(\mathcal{D}_{T_7}).
			\end{aligned}
		\end{equation}
	\end{thm}
	
	It is a fact that every pure contraction is c.n.u. In case of $T_7$ is pure, we note that $\Theta_{T_7}$ is an inner function, which implies that $M_{\Theta_{T_7}}$ is an isometry. On the other hand, $\mathcal{Q}_{T^*_7} = 0$ and hence $\overline{\triangle_{T_7} L^2(\mathcal{D}_{T_7})} = 0$. This implies that $\mathcal{H}^{\textbf{T}}_{NF} = H^2(\mathcal{D}_{T^*_7}) \ominus \Theta_{T_7}H^2(\mathcal{D}_{T_7}) = \mathcal{H}_{T_7}$ [See Section $3$, \cite{apal3}]. Thus, for this case, Nagy-Foias model for c.n.u. $\Gamma_{E(3; 3; 1, 1, 1)}$-contraction reduced to the functional model for pure $\Gamma_{E(3; 3; 1, 1, 1)}$-contraction. We only state the model in the following theorem.
	
	\begin{thm}[Theorem $3.2$, \cite{apal3}]\label{Thm 7}
		Let $\textbf{T} = (T_1, \dots, T_7)$ be a pure $\Gamma_{E(3; 3; 1, 1, 1)}$-contraction on a Hilbert space $\mathcal{H}$. Suppose that $\tilde{F}_i, 1\leq i \leq 6$ are fundamental operators of $\textbf{T}^* = (T^*_1, \dots, T^*_7)$. Then
		\begin{enumerate}\label{Model 1}
			\item $T_i$  is unitarily equivalent to  $P_{\mathcal{H}_{T_7}}(I \otimes\tilde{ F}^*_i + M_z \otimes \tilde{F}_{7-i})_{|_{\mathcal{H}_{T_7}}}$ for $1 \leqslant i \leqslant 6$, and 
			
			\item $T_7$  is unitarily equivalent to $P_{\mathcal{H}_{T_7}}(M_z \otimes I_{\mathcal{D}_{T^*_7}})_{|_{\mathcal{H}_{T_7}}}$,
		\end{enumerate} where $\mathcal{H}_{T_7} =
		(H^2(\mathbb{D}) \otimes \mathcal{D}_{T^*_7}) \ominus M_{\Theta_{T_7}}(H^2(\mathbb{D}) \otimes \mathcal{D}_{T_7}).$
	\end{thm}
	
	Since, for the case $T_7$ pure, $\mathcal{Q}_{T^*_7} = 0$, then $\mathcal{H}^{\textbf{T}}_D = \Ran \mathcal{O}_{D_{T^*_7}, T^*_7}$. Therefore, the Douglas model for $\Gamma_{E(3; 3; 1, 1, 1)}$-contraction reduced to the following one.
	
	\begin{thm}\label{Thm 7A}
		Let $\textbf{T} = (T_1, \dots, T_7)$ be a pure $\Gamma_{E(3; 3; 1, 1, 1)}$-contraction on a Hilbert space $\mathcal{H}$. Suppose that $\tilde{F}_i, 1\leq i \leq 6$ are fundamental operators of $\textbf{T}^* = (T^*_1, \dots, T^*_7)$. Then
		\begin{enumerate}\label{Model 1A}
			\item $T_i$  is unitarily equivalent to  $P_{\Ran \mathcal{O}_{D_{T^*_7}, T^*_7}}(I \otimes\tilde{ F}^*_i + M_z \otimes \tilde{F}_{7-i})_{|_{\Ran \mathcal{O}_{D_{T^*_7}, T^*_7}}}$ for $1 \leqslant i \leqslant 6$, and 
			
			\item $T_7$  is unitarily equivalent to $P_{\Ran \mathcal{O}_{D_{T^*_7}, T^*_7}}(M_z \otimes I_{\mathcal{D}_{T^*_7}})_{|_{\Ran \mathcal{O}_{D_{T^*_7}, T^*_7}}}$,
		\end{enumerate}
	\end{thm}
	
	Analogously, we describe the Nagy-Foias model for c.n.u. $\Gamma_{E(3; 2; 1, 2)}$-contraction. Let us consider $\textbf{S} = (S_1, S_2, S_3, \tilde{S}_1, \tilde{S}_2)$ is a $\Gamma_{E(3; 2; 1, 2)}$-contraction on a Hilbert space $\mathcal{H}$ such that $S_3$ is c.n.u. contraction. Consider the function
	\begin{equation}\label{Delta_{S_3}}
		\begin{aligned}
			\triangle_{S_3}(\omega) &=
			(I - \Theta_{S_3}(\omega)^*\Theta_{S_3}(\omega))^{1/2}.
		\end{aligned}
	\end{equation}
	Due to Sz.-Nagy and Foias,
	\begin{equation}\label{V^S_NF}
		\begin{aligned}
			V^{\textbf{S}}_{NF} &:=
			\begin{pmatrix}
				M_z & 0\\
				0 & M_{\omega}|_{\overline{\triangle_{S_3} L^2(\mathcal{D}_{S_3})}}
			\end{pmatrix} :
			\begin{pmatrix}
				H^2(\mathcal{D}_{S^*_3})\\
				\overline{\triangle_{S_3} L^2(\mathcal{D}_{S_3})}
			\end{pmatrix} \to
			\begin{pmatrix}
				H^2(\mathcal{D}_{S^*_3})\\
				\overline{\triangle_{S_3} L^2(\mathcal{D}_{S_3})}
			\end{pmatrix}
		\end{aligned}
	\end{equation}
	is a minimal isometric dilation of $S_3$ via the isometric embedding
	\begin{equation}\label{Pi^S_NF}
		\begin{aligned}
			\Pi^{\textbf{S}}_{NF} : \mathcal{H} \to
			\begin{pmatrix}
				H^2(\mathcal{D}_{S^*_3})\\
				\overline{\triangle_{S_3} L^2(\mathcal{D}_{S_3})}
			\end{pmatrix}
		\end{aligned}
	\end{equation}
	such that
	\begin{equation}\label{H^S_NF}
		\begin{aligned}
			\mathcal{H}^{\textbf{S}}_{NF} &= \Ran \Pi^{\textbf{S}}_{NF} =
			\begin{pmatrix}
				H^2(\mathcal{D}_{S^*_3})\\
				\overline{\triangle_{S_3} L^2(\mathcal{D}_{S_3})}
			\end{pmatrix} \ominus
			\begin{pmatrix}
				\Theta_{S_3}\\
				\triangle_{S_3}
			\end{pmatrix}
			H^2(\mathcal{D}_{S_3}).
		\end{aligned}
	\end{equation}
	Now we define $M^{(1)}_{NF}, M^{(2)}_{NF}, \tilde{M}^{(1)}_{NF}, \tilde{M}^{(2)}_{NF}$ on $\overline{\triangle_{S_3} L^2(\mathcal{D}_{S_3})}$ by
	\begin{equation}\label{M^i_NF, M^i_NF tilda}
		\begin{aligned}
			M^{(i)}_{NF} &= u^{\textbf{S}}_{\min}M^{(i)}_Du^{\textbf{S}*}_{\min}, \,\, \text{and} \,\, \tilde{M}^{(j)}_{NF} = u^{\textbf{S}}_{\min}\tilde{M}^{(j)}_Du^{\textbf{S}*}_{\min} \,\, \text{for} \,\, 1 \leqslant i, j \leqslant 2.
		\end{aligned}
	\end{equation}
	Similar to the case of $\Gamma_{E(3; 3; 1, 1, 1)}$-contraction, there exists a unitary $u^{\textbf{S}}_{\min} : \mathcal{Q}_{S^*_3} \to \overline{\triangle_{S_3} L^2(\mathcal{D}_{S_3})}$ such that
	\begin{equation}\label{u^S_min, M^3_D, M_w}
		\begin{aligned}
			u^{\textbf{S}}_{\min}M^{(3)}_D &= M_{\omega}|_{\overline{\triangle_{S_3} L^2(\mathcal{D}_{S_3})}}u^{\textbf{S}}_{\min}
		\end{aligned}
	\end{equation}
	$$\rm{and}$$
	\begin{equation}\label{Pi^S_NF, Pi^S_D}
		\begin{aligned}
			\Pi^{\textbf{S}}_{NF} &=
			\begin{pmatrix}
				I_{H^2(\mathcal{D}_{S^*_3})} & 0\\
				0 & u^{\textbf{S}}_{\min}
			\end{pmatrix}
			\Pi^{\textbf{S}}_D.
		\end{aligned}
	\end{equation}
	Therefore, the following theorem on functional on functional model for c.n.u. $\Gamma_{E(3; 2; 1, 2)}$-contraction is straightforward consequence of \eqref{Pi^S_NF, Pi^S_D} and Theorem \ref{Douglas Model 2}.
	
	\begin{thm}[Sz.-Nagy-Foias Model for C.N.U. $\Gamma_{E(3; 2; 1, 2)}$-Contraction]\label{Nagy-Foias Model 2}
		Let $\textbf{S} = (S_1, S_2, S_3, \tilde{S}_1, \tilde{S}_2)$ be a c.n.u. $\Gamma_{E(3; 2; 1, 2)}$-contraction on a Hilbert space $\mathcal{H}$ and $\hat{G}_1, 2\hat{G}_2, 2\hat{\tilde{G}}_1, \hat{\tilde{G}}_2$ be the fundamental operators of $\textbf{S}^*$. Then $\textbf{S}$ is unitarily equivalent to
		\begin{equation}
			\begin{aligned}
				&P_{\mathcal{H}^{\textbf{S}}_{NF}}
				\left(
				\begin{pmatrix}
					M_{\hat{G}^*_1 + \hat{\tilde{G}}_2z} & 0\\
					0 & M^{(1)}_{NF}
				\end{pmatrix},
				\begin{pmatrix}
					M_{2\hat{G}^*_2 + 2\hat{\tilde{G}}_1z} & 0\\
					0 & M^{(2)}_{NF}
				\end{pmatrix},
				\begin{pmatrix}
					M_z & 0\\
					0 & M^{(3)}_{NF}
				\end{pmatrix}, \right.\\
				&\left. \hspace{2cm}
				\begin{pmatrix}
					M_{2\hat{\tilde{G}}^*_1 + 2\hat{G}_2z} & 0\\
					0 & \tilde{M}^{(1)}_{NF}
				\end{pmatrix}
				\begin{pmatrix}
					M_{\hat{\tilde{G}}^*_2 + \hat{G}_1z} & 0\\
					0 & \tilde{M}^{(2)}_{NF}
				\end{pmatrix}
				\right)\Bigg|_{\mathcal{H}^{\textbf{S}}_{NF}}
			\end{aligned}
		\end{equation}
		($\hat{G}_1, 2\hat{G}_2, 2\hat{\tilde{G}}_1, \hat{\tilde{G}}_2$ satisfy \eqref{Commutative B(i)} and \eqref{Commutative B(ii)} when \eqref{Douglas Model 2.1} is commutative) where $\mathcal{H}^{\textbf{S}}_{NF}$ is the functional model space of $\textbf{S}$ defined by
		\begin{equation}
			\begin{aligned}
				\mathcal{H}^{\textbf{S}}_{NF} &:=
				\Ran \Pi^{\textbf{S}}_{NF} =
				\begin{pmatrix}
					H^2(\mathcal{D}_{S^*_3})\\
					\overline{\triangle_{S_3} L^2(\mathcal{D}_{S_3})}
				\end{pmatrix} \ominus
				\begin{pmatrix}
					\Theta_{S_3}\\
					\triangle_{S_3}
				\end{pmatrix}
				H^2(\mathcal{D}_{S_3}).
			\end{aligned}
		\end{equation}
	\end{thm}
	
	In case of $S_3$ a pure contraction, $\Theta_{S_3}$ is an inner function and hence, $M_{\Theta_{S_3}}$ is an isometry. Since, $S_3$ is pure it is clear that $\mathcal{Q}_{S^*_3} = 0$ and thus, $\overline{\triangle_{S_3} L^2(\mathcal{D}_{S_3})} = 0$ [See Section $3$, \cite{apal3}]. It therefore implies that $\mathcal{H}^{\textbf{S}}_{NF} = H^2(\mathcal{D}_{S^*_3}) \ominus \Theta_{S_3}H^2(\mathcal{D}_{S_3}) = \mathcal{H}_{S_3}$. Therefore, the functional model for pure $\Gamma_{E(3; 2; 1, 2)}$-contraction is immediate from Nagy-Foias model for c.n.u. $\Gamma_{E(3; 2; 1, 2)}$-contraction.
	
	Let $\hat{A}_1= P_{\mathcal{H}_{S_3}}(I \otimes \hat{G}^*_1 + M_z \otimes \hat{ \tilde{G}}_2)_{|_{\mathcal{H}_{S_3}}}, \hat{A}_2=P_{\mathcal{H}_{S_3}}(I \otimes 2\hat{G}^*_2 + M_z \otimes 2\hat{\tilde{G}}_1)_{|_{\mathcal{H}_{S_3}}}, \hat{A}_3=P_{\mathcal{H}_{S_3}}(M_z \otimes I_{\mathcal{D}_{S^*_3}})_{|_{\mathcal{H}_{S_3}}},\hat{B}_1=P_{\mathcal{H}_{S_3}}(I \otimes \hat{ \tilde{G}}^*_2 + M_z \otimes \hat{G}_1)_{|_{\mathcal{H}_{S_3}}},\hat{B}_2=P_{\mathcal{H}_{S_3}}(I \otimes \hat{ \tilde{G}}^*_2 + M_z \otimes \hat{G}_1)_{|_{\mathcal{H}_{S_3}}},$ where $\mathcal{H}_{S_3} =(H^2(\mathbb{D}) \otimes \mathcal{D}_{S^*_3}) \ominus M_{\Theta_{S_3}}(H^2(\mathbb{D}) \otimes \mathcal{D}_{S_3}).$ We then only state the following theorem from \cite{apal3}.
	
	\begin{thm}[Theorem $3.6$, \cite{apal3}]\label{Thm 8}
		Let $\textbf{S} = (S_1, S_2, S_3, \tilde{S}_1, \tilde{S}_2)$ be a pure $\Gamma_{E(3; 2; 1, 2)}$-contraction on a Hilbert space $\mathcal{H}$. Let $\hat{G}_1, 2\hat{G}_2, 2\hat{\tilde{G}}_1, \hat{\tilde{G}}_2$ be fundamental operators for $\textbf{S}^* = (S^*_1, S^*_2, S^*_3, \tilde{S}^*_1, \tilde{S}^*_2)$. Then 
		\begin{enumerate}\label{Model 2}
			\item $S_1$ is unitarily equivalent to $\hat{A}_1$,
			
			\item $S_2$  is unitarily equivalent to $\hat{A}_2$,
			
			\item $S_3$  is unitarily equivalent to $ \hat{A}_3$,
			
			\item $\tilde{S}_1$  is unitarily equivalent to $\hat{B}_1$,
			
			\item $\tilde{S}_2$  is unitarily equivalent to $\hat{B}_2$.
		\end{enumerate}
	\end{thm}
	
	When $S_3$ is pure then $\mathcal{Q}_{S^*_3} = 0$ and hence, $\mathcal{H}^{\textbf{S}}_D = \Ran \mathcal{O}_{D_{S^*_3}, S^*_3}$. Suppose that\\
	$\tilde{A}_1= P_{\Ran \mathcal{O}_{D_{S^*_3}, S^*_3}}(I \otimes \hat{G}^*_1 + M_z \otimes \hat{ \tilde{G}}_2)_{|_{\Ran \mathcal{O}_{D_{S^*_3}, S^*_3}}}, \tilde{A}_2=P_{\Ran \mathcal{O}_{D_{S^*_3}, S^*_3}}(I \otimes 2\hat{G}^*_2 + M_z \otimes 2\hat{\tilde{G}}_1)_{|_{\Ran \mathcal{O}_{D_{S^*_3}, S^*_3}}}, \tilde{A}_3=P_{\Ran \mathcal{O}_{D_{S^*_3}, S^*_3}}(M_z \otimes I_{\mathcal{D}_{S^*_3}})_{|_{\Ran \mathcal{O}_{D_{S^*_3}, S^*_3}}}, \tilde{B}_1=P_{\Ran \mathcal{O}_{D_{S^*_3}, S^*_3}}(I \otimes \hat{ \tilde{G}}^*_2 + M_z \otimes \hat{G}_1)_{|_{\Ran \mathcal{O}_{D_{S^*_3}, S^*_3}}},\tilde{B}_2=P_{\Ran \mathcal{O}_{D_{S^*_3}, S^*_3}}(I \otimes \hat{ \tilde{G}}^*_2 + M_z \otimes \hat{G}_1)_{|_{\Ran \mathcal{O}_{D_{S^*_3}, S^*_3}}}$. Consequently, the Douglas model for $\Gamma_{E(3; 2; 1, 2)}$-contraction is reduced to the following one.
	
	\begin{thm}\label{Thm 8A}
		Let $\textbf{S} = (S_1, S_2, S_3, \tilde{S}_1, \tilde{S}_2)$ be a pure $\Gamma_{E(3; 2; 1, 2)}$-contraction on a Hilbert space $\mathcal{H}$. Let $\hat{G}_1, 2\hat{G}_2, 2\hat{\tilde{G}}_1,\\ \hat{\tilde{G}}_2$ be fundamental operators for $\textbf{S}^* = (S^*_1, S^*_2, S^*_3, \tilde{S}^*_1, \tilde{S}^*_2)$. Then 
		\begin{enumerate}\label{Model 2A}
			\item $S_1$ is unitarily equivalent to $\tilde{A}_1$,
			
			\item $S_2$  is unitarily equivalent to $\tilde{A}_2$,
			
			\item $S_3$  is unitarily equivalent to $ \tilde{A}_3$,
			
			\item $\tilde{S}_1$  is unitarily equivalent to $\tilde{B}_1$,
			
			\item $\tilde{S}_2$  is unitarily equivalent to $\tilde{B}_2$.
		\end{enumerate}
	\end{thm}
	
	\begin{rem}\label{Rem 2}
		From Theorem \ref{Thm 7} it is clear that a pure $\Gamma_{E(3; 3; 1, 1, 1)}$-contraction dilates to a pure $\Gamma_{E(3; 3; 1, 1, 1)}$-isometry. In a similar way, Theorem \ref{Thm 8} implies that a pure $\Gamma_{E(3; 2; 1, 2)}$-contraction dilates to a pure $\Gamma_{E(3; 2; 1, 2)}$-isometry.
	\end{rem}

	\section{Sch\"{a}ffer Type Isometric Dilation of $\Gamma_{E(3; 3; 1, 1, 1)}$-Contraction and $\Gamma_{E(3; 2; 1, 2)}$-Contraction: A Model Theoretic Approach}\label{Schaffer Type Model}
	
	Sch\"{a}ffer's construction of isometric dilation of a contraction can be found in \cite{Schaffer}. We develop a Sch\"{a}ffer type model for $\Gamma_{E(3; 3; 1, 1, 1)}$-isometric dilation and $\Gamma_{E(3; 2; 1, 2)}$-isometric dilation.
	
	Let
	\begin{equation}\label{V^i_Sc, V^7_Sc}
		\begin{aligned}
			V^{(i)}_{Sc} &=
			\begin{pmatrix}
				T_i & 0\\
				F^*_{7-i}D_{T_7} & M_{F_i + F^*_{7-i}z}
			\end{pmatrix} \,\, \text{for} \,\, 1 \leqslant i \leqslant 6 \,\, \text{and} \,\,
			V^{(7)}_{Sc} =
			\begin{pmatrix}
				T_7 & 0\\
				D_{T_7} & M_z
			\end{pmatrix}.
		\end{aligned}
	\end{equation}
	We demonstrate that $\textbf{V}_{Sc} = (V^{(1)}_{Sc}, \dots, V^{(6)}_{Sc}, V^{(7)}_{Sc})$ is a $\Gamma_{E(3; 3; 1, 1, 1)}$-isometric dilation on the model space
	$\begin{pmatrix}
		\mathcal{H}\\
		H^2(\mathcal{D}_{T_7})
	\end{pmatrix}$
	with respect to the isometry
	\begin{equation}\label{Pi^T_Sc}
		\begin{aligned}
			\Pi^{\textbf{T}}_{Sc} : \mathcal{H} \to 
			\begin{pmatrix}
				\mathcal{H}\\
				H^2(\mathcal{D}_{T_7})
			\end{pmatrix} \,\, \text{defined by} \,\,
			\Pi^{\textbf{T}}_{Sc} h = (h, 0) \,\, \text{for all} \,\, h \in \mathcal{H}.
		\end{aligned}
	\end{equation}
	
	\begin{thm}[Sch\"{a}ffer Isometric Dilation for $\Gamma_{E(3; 3; 1, 1, 1)}$-Contraction]\label{Schaffer Model 1}
		Let $\textbf{T} = (T_1, \dots, T_7)$ be a $\Gamma_{E(3; 3; 1, 1, 1)}$-contraction on a Hilbert space $\mathcal{H}$. Then
		\begin{equation}\label{Schaffer Model 1.1}
			\begin{aligned}
				\textbf{V}_{Sc} &= (V^{(1)}_{Sc}, \dots, V^{(6)}_{Sc}, V^{(7)}_{Sc})
			\end{aligned}
		\end{equation}
		is the $\Gamma_{E(3; 3; 1, 1, 1)}$-isometric (Sch\"{a}ffer type) dilation of $\textbf{T}$ on the model space
		$\begin{pmatrix}
			\mathcal{H}\\
			H^2(\mathcal{D}_{T_7})
		\end{pmatrix}$ satisfies
		\begin{equation}
			\begin{aligned}
				\Pi^{\textbf{T}}_{Sc}T^*_i &=
				V^{(i)*}_{Sc}\Pi^{\textbf{T}}_{Sc}
			\end{aligned}
		\end{equation}
		and that is uniquely determined by the operators $F_1, \dots, F_6 \in  \mathcal{B}(\mathcal{D}_{T_7})$ satisfying \eqref{Fundamental 1}, \eqref{Commutative 1} and $w(F_i + F^*_{7-i}z) \leqslant 1$ for $1 \leqslant i \leqslant 6$ and $z \in \mathbb{T}$.
		
		Conversely, let $\textbf{T} = (T_1, \dots, T_7)$ be commuting $7$-tuple of bounded operators, acting on a Hilbert space $\mathcal{H}$ such that $||T_i|| \leqslant 1$ and satisfies \eqref{Fundamental Eq}, \eqref{Commutative 1} for some $F_1, \dots, F_6 \in \mathcal{B}(\mathcal{D}_{T_7})$ with $w(F_i + F^*_{7-i}z) \leqslant 1$ for all $z \in \mathbb{T}$. Then $\textbf{T}$ is a $\Gamma_{E(3; 3; 1, 1, 1)}$-contraction.
	\end{thm}
	
	\begin{proof}
		Let $\textbf{T}$ be a $\Gamma_{E(3; 3; 1, 1, 1)}$-contraction on $\mathcal{H}$. Then by [Theorem $6.3$, \cite{apal2}] we have that $\mathcal{H}$ can be decomposed into a direct sum $\mathcal{H} = \mathcal{H}_u \oplus \mathcal{H}_{cnu}$ of two closed proper subspaces $\mathcal{H}_u$ and $\mathcal{H}_{cnu}$ such that $\textbf{T}|_{\mathcal{H}_{cnu}}$ is a c.n.u. $\Gamma_{E(3; 3; 1, 1, 1)}$-contraction and $\mathcal{H}_u$ is a $\Gamma_{E(3; 3; 1, 1, 1)}$-unitary. Since, $\textbf{T}|_{\mathcal{H}_{cnu}}$ is a c.n.u. $\Gamma_{E(3; 3; 1, 1, 1)}$-contraction then $T_7|_{\mathcal{H}_{cnu}}$ is a c.n.u. contraction. By factorization of dilation [Theorem $4.1$, \cite{Sarkar}], there exists an isometry $\Phi : \mathcal{H}^{\textbf{T}}_{NF} \to \begin{pmatrix}
			\mathcal{H}\\
			H^2(\mathcal{D}_{T_7})
		\end{pmatrix}$ such that
		\begin{equation}\label{SM 1.2}
			\begin{aligned}
				\Pi^{\textbf{T}}_{Sc} &= \Phi \Pi^{\textbf{T}}_{NF}.
			\end{aligned}
		\end{equation}
		Since Sch\"{a}ffer dilation is minimal, then it follows that $\Phi$ is a unitary. Let $\textbf{V} = (V_1, \dots, V_6, V_7)$ be the isometric dilation of $\textbf{T}$. Then
		\begin{equation}\label{SM 1.3}
			\begin{aligned}
				&(V_1, \dots, V_6, V_7) =
				\Phi\left(
				\begin{pmatrix}
					M_{F^*_1 + F_6z} & 0\\
					0 & N^{(1)}_{NF}
				\end{pmatrix}, \dots,
				\begin{pmatrix}
					M_{F^*_6 + F_1z} & 0\\
					0 & N^{(6)}_{NF}
				\end{pmatrix},
				\begin{pmatrix}
					M_z & 0\\
					0 & M_{e^{it}}|_{\overline{\triangle_{T_7} L^2(\mathcal{D}_{T_7})}}
				\end{pmatrix}\right)\Phi^*,
			\end{aligned}
		\end{equation}
		where $(N^{(1)}_{NF}, \dots, N^{(6)}_{NF}, M_{e^{it}}|_{\overline{\triangle_{T_7} L^2(\mathcal{D}_{T_7})}})$ is a $\Gamma_{E(3; 3; 1, 1, 1)}$-unitary on $\overline{\triangle_{T_7}L^2(\mathcal{D}_{T_7})}$ and
		\begin{equation}\label{SM 1.4}
			\begin{aligned}
				\Pi^{\textbf{T}}_{Sc}T^*_7
				&= V^*_7\Pi^{\textbf{T}}_S \,\, \text{and} \,\, \Pi^{\textbf{T}}_ST^*_i = V^*_i\Pi^{\textbf{T}}_{Sc} \,\,  \text{for} \,\, 1 \leqslant i \leqslant 6.
			\end{aligned}
		\end{equation}
		Now taking the direct sum of $\textbf{T}|_{\mathcal{H}_u}$ with $\textbf{T}|_{\mathcal{H}_{cnu}}$, we conclude that $\Pi^{\textbf{T}}_{Sc}$ is a minimal isometric dilation of $\textbf{T}$. By uniqueness of Sch\"{a}ffer dilation we have that
		\begin{equation}\label{SM 1.5}
			\begin{aligned}
				V_7 &=
				\begin{pmatrix}
					T_7 & 0\\
					D_{T_7} & M_z
				\end{pmatrix} = V^{(7)}_{Sc}.
			\end{aligned}
		\end{equation}
		We show that
		\begin{equation}\label{SM 1.6}
			\begin{aligned}
				V_i &=
				\begin{pmatrix}
					T_i & 0\\
					F^*_{7-i}D_{T_7} & M_{F_i + F^*_{7-i}z}
				\end{pmatrix} = V^{(i)}_{Sc}
			\end{aligned}
		\end{equation}
		for some $F_1, \dots, F_6 \in \mathcal{B}(\mathcal{D}_{T_7})$ satisfying $w(F_i + F^*_{7-i}z) \leqslant 1$ for $1 \leqslant i \leqslant 6$ and $z \in \mathbb{T}$. Let us assume
		\begin{equation*}
			\begin{aligned}
				V_i &=
				\begin{pmatrix}
					T_i & 0\\
					A^{(i)}_{21} & A^{(i)}_{22}
				\end{pmatrix}
				\,\, \text{for} \,\, 1 \leqslant i \leqslant 6.
			\end{aligned}
		\end{equation*}
		Since $\textbf{V}$ is a $\Gamma_{E(3; 3; 1, 1, 1)}$-isometry then by [Theorem $4.4$, \cite{apal2}] we have $V_i = V^*_{7-i}V_7$ for $1 \leqslant i \leqslant 6$, i.e.,
		\begin{equation}\label{SM 1.7}
			\begin{aligned}
				\begin{pmatrix}
					T_i & 0\\
					A^{(i)}_{21} & A^{(i)}_{22}
				\end{pmatrix} &=
				\begin{pmatrix}
					T^*_{7-i} & A^{(7-i)*}_{21}\\
					0 & A^{(i)*}_{22}
				\end{pmatrix}
				\begin{pmatrix}
					T_7 & 0\\
					D_{T_7} & M_z
				\end{pmatrix}\\
				&=
				\begin{pmatrix}
					T^*_{7-i}T_7 + A^{(7-i)*}_{21}D_{T_7} & A^{(7-i)*}_{21}M_z\\
					A^{(7-i)*}_{22}D_{T_7} & A^{(7-i)*}_{22}M_z
				\end{pmatrix}.
			\end{aligned}
		\end{equation}
		From \eqref{SM 1.7} we have
		\begin{equation}\label{SM 1.8}
			\begin{aligned}
				T_i = T^*_{7-i}T_7 + A^{(7-i)*}_{21}D_{T_7}, A^{(7-i)*}_{21}M_z = 0, A^{(i)}_{21} = A^{(7-i)*}_{22}D_{T_7} \,\, \text{and} \,\,
				A^{(i)}_{22} = A^{(7-i)*}_{22}M_z.
			\end{aligned}
		\end{equation}
		Also by the commutativity of $V_i$ with $V_7$ we get $A^{(i)}_{22}M_z = M_zA^{(i)}_{22}$. This implies that there exists $\Psi_i \in \mathcal{B}(\mathcal{D}_{T_7})$ such that $A^{(i)}_{22} = M_{\Psi_i}$ for $1 \leqslant i \leqslant 6$. Then proceeding similarly as in Theorem \ref{Wold decomposition 1} we see that $\Psi_i$'s are of the form $\Psi_i(z) = F_i + F^*_{7-i}z$ for some $F_i \in \mathcal{B}(\mathcal{D}_{T_7})$ for $1 \leqslant i \leqslant 6$ and $z \in \mathbb{T}$. Again from \eqref{SM 1.8} we get
		\begin{equation}\label{SM 1.9}
			\begin{aligned}
				A^{(i)}_{21}
				&= A^{(7-i)*}_{22}D_{T_7}
				= M^*_{\Psi_{7-i}}D_{T_7}
				= M^*_{F_{7-i} + F^*_iz}D_{T_7}
				= (F^*_{7-i} + F_iM^*_z)D_{T_7}
				= F^*_{7-i}D_{T_7}.
			\end{aligned}
		\end{equation}
		Thus from $T_i = T^*_{7-i}T_7 + A^{(7-i)*}_{21}D_{T_7}$ in \eqref{SM 1.8} and \eqref{SM 1.9} we obtain the following:
		\begin{equation}\label{SM 1.10}
			\begin{aligned}
				T_i - T^*_{7-i}T_7 &= D_{T_7}F_iD_{T_7} \,\, \text{for} \,\, 1 \leqslant i \leqslant 6.
			\end{aligned}
		\end{equation}
		Hence $F_1, \dots, F_6$ are the fundamental operators of $\textbf{T}$ and hence $F_i$'s are unique proving that $V_i$'s are of the form that in \eqref{SM 1.6} and that are uniquely determined. And since $F_1, \dots, F_6$ are the fundamental operators of $\textbf{T}$, we have $w(F_i + F^*_{7-i}z) \leqslant 1$ for $1 \leqslant i \leqslant 6$. And hence $\Pi^{\textbf{T}}_{Sc}T^*_i =
		V^{(i)*}_{Sc}\Pi^{\textbf{T}}_{Sc}$.
		
		Conversely, let $\textbf{T} = (T_1, \dots, T_7)$ be a commuting $7$-tuple of bounded operators on $\mathcal{H}$ such that $||T_i|| \leqslant 1$ and $T_i - T^*_{7-i}T_7 = D_{T_7}F_iD_{T_7}$ for some $F_i \in \mathcal{B}(\mathcal{D}_{T_7})$ with $w(F_i + F^*_{7-i}z) \leqslant 1$ for $1 \leqslant i \leqslant 6$ and $z \in \mathbb{T}$. Assume that the Sch\"{a}ffer dilation $\Pi^{\textbf{T}}_S$ satisfies
		\begin{equation*}
			\begin{aligned}
				\Pi^{\textbf{T}}_{Sc}T^*_7
				&= V^{(7)*}_{Sc}\Pi^{\textbf{T}}_{Sc} \,\, \text{and} \,\, \Pi^{\textbf{T}}_{Sc}T^*_i = V^{(i)*}_{Sc}\Pi^{\textbf{T}}_{Sc} \,\,  \text{for} \,\, 1 \leqslant i \leqslant 6,
			\end{aligned}
		\end{equation*}
		where $V^{(i)}_{Sc}$ is the operator
		$\begin{pmatrix}
			T_i & 0\\
			F^*_{7-i}D_{T_7} & M_{F_i + F^*_{7-i}z}
		\end{pmatrix}$ for $1 \leqslant i \leqslant 6$. Then one can easily check that $V^{(i)}_{Sc} = V^{(7-i)*}_{Sc}V^{(7)}_{Sc}$ for $1 \leqslant i \leqslant 6$. Since $V^{(i)}_{Sc}$ commutes with $V^{(7)}_{Sc}$ and $V^{(i)}_{Sc} = V^{(7-i)*}_{Sc}V^{(7)}_{Sc}$ then $V^{(i)}_{Sc}$ are hyponormal operators and then by [Theorem $1$, \cite{Stampfli}] we have that $||V^{(i)}_{Sc}|| = r(V^{(i)}_{Sc})$ for $1 \leqslant i \leqslant 6$. We show that $r(V^{(i)}_{Sc}) \leqslant 1$ Now by [Lemma $1$, \cite{Hong}], $\sigma(V^{(i)}_{Sc}) \subseteq \sigma(T_i) \cup \sigma(M_{F_i + F^*_{7-i}z})$. Since $||T_i|| \leqslant 1$ then $r(T_i) \leqslant 1$. Since $r(V^{(i)}_{Sc}) \leqslant w(V^{(i)}_{Sc})$ then we just show $w(V^{(i)}_{Sc}) \leqslant 1$. By the similar method in [Theorem $4.6$, \cite{apal2}] we obtain $w(V^{(i)}_{Sc}) \leqslant 1$. It now follows from here that $\textbf{V}_{Sc}$ is a $\Gamma_{E(3; 3; 1, 1, 1)}$-isometry; and as $\textbf{T}$ is the restriction of $\textbf{V}_{Sc}$ on $\mathcal{H}$ hence $\textbf{T}$ is a $\Gamma_{E(3; 3; 1, 1, 1)}$-contraction. This finishes the proof.
	\end{proof}
	
	We next produce an analogous dilation for $\Gamma_{E(3; 2; 1, 2)}$-contraction. Let us consider the operators $W^{(1)}_{Sc}, W^{(2)}_{Sc},\\ W^{(3)}_{Sc}, \tilde{W}^{(1)}_{Sc}, \tilde{W}^{(2)}_{Sc}$ as follows
	\begin{equation}
		\begin{aligned}
			&W^{(1)}_{Sc} =
			\begin{pmatrix}
				S_1 & 0\\
				\tilde{G}^*_2D_{S_3} & M_{G_1 + \tilde{G}^*_2z}
			\end{pmatrix},
			W^{(2)}_{Sc} =
			\begin{pmatrix}
				S_2 & 0\\
				2\tilde{G}^*_1D_{S_3} & M_{2G_2 + 2\tilde{G}^*_1z}
			\end{pmatrix},
			W^{(3)}_{Sc} =
			\begin{pmatrix}
				S_3 & 0\\
				D_{S_3} & M_z
			\end{pmatrix},\\
			&\hspace{2cm}
			\tilde{W}^{(1)}_{Sc} =
			\begin{pmatrix}
				\tilde{S}_1 & 0\\
				2G^*_2D_{S_3} & M_{2\tilde{G}_1 + 2G^*_2z}
			\end{pmatrix},
			\tilde{W}^{(2)}_{Sc} =
			\begin{pmatrix}
				\tilde{S}_2 & 0\\
				G^*_1D_{S_3} & M_{\tilde{G}_2 + G^*_1z}
			\end{pmatrix}.
		\end{aligned}
	\end{equation}
	We prove that $\textbf{W}_{Sc} = (W^{(1)}_{Sc}, W^{(2)}_{Sc}, W^{(3)}_{Sc}, \tilde{W}^{(1)}_{Sc}, \tilde{W}^{(2)}_{Sc})$ is an $\Gamma_{E(3; 2; 1, 2)}$-isometric dilation of $\textbf{S}$ acting on the model space
	$
	\begin{pmatrix}
		\mathcal{H}\\
		H^2(\mathcal{D}_{S_3})
	\end{pmatrix}$ with respect to the isometry
	\begin{equation}\label{Pi^S_Sc}
	\begin{aligned}
		\Pi^{\textbf{S}}_{Sc} : \mathcal{H} \to 
		\begin{pmatrix}
			\mathcal{H}\\
			H^2(\mathcal{D}_{S_3})
		\end{pmatrix} \,\, \text{defined by} \,\,
		\Pi^{\textbf{S}}_{Sc} h = (h, 0) \,\, \text{for all} \,\, h \in \mathcal{H}.
	\end{aligned}
	\end{equation}
	
	We skip the proof of the following theorem as it is identical with the proof of Theorem \ref{Schaffer Model 1}.
	
	\begin{thm}[Sch\"{a}ffer Isometric Dilation for $\Gamma_{E(3; 2; 1, 2)}$-Contraction]\label{Schaffer Model 2}
		Let $\textbf{S} = (S_1, S_2, S_3, \tilde{S}_1, \tilde{S}_2)$ be a $\Gamma_{E(3; 2; 1, 2)}$-contraction on a Hilbert space $\mathcal{H}$. Then
		\begin{equation}\label{Schaffer Model 2.1}
			\begin{aligned}
				\textbf{W}_{Sc} = (W^{(1)}_{Sc}, W^{(2)}_{Sc}, W^{(3)}_{Sc}, \tilde{W}^{(1)}_{Sc}, \tilde{W}^{(2)}_{Sc})
			\end{aligned}
		\end{equation}
		is the Sch\"{a}ffer isometric dilation of $\textbf{S}$ on the model space
		$
		\begin{pmatrix}
			\mathcal{H}\\
			H^2(\mathcal{D}_{S_3})
		\end{pmatrix}$
		satisfies
		\begin{equation}
			\begin{aligned}
				\Pi^{\textbf{S}}_{Sc}S^*_i &=
				W^{(i)*}_{Sc}\Pi^{\textbf{S}}_{Sc} \,\, \text{and} \,\,
				\Pi^{\textbf{S}}_{Sc}\tilde{S}^*_j &=
				\tilde{W}^{(j)*}_{Sc}\Pi^{\textbf{S}}_{Sc}
			\end{aligned}
		\end{equation}
		for $1 \leqslant i, j \leqslant 2$ and that is uniquely determined by the operators $G_1, 2G_2, 2\tilde{G}_1, \tilde{G}_2 \in  \mathcal{B}(\mathcal{D}_{S_3})$ satisfying \eqref{funda1}, \eqref{funda11}, \eqref{Commutative 2(i)}, \eqref{Commutative 2(ii)} and $w(G_1 + \tilde{G}^*_2z) \leqslant 1$ and $w(G_2 + \tilde{G}^*_1z) \leqslant 1$ for all $z \in \mathbb{T}$.
		
		Conversely, let $\textbf{S} = (S_1, S_2, S_3, \tilde{S}_1, \tilde{S}_2)$ be commuting $5$-tuple of bounded operators, acting on a Hilbert space $\mathcal{H}$ such that $||S_1|| \leqslant 1, ||\tilde{S}_2||, ||S_2|| \leqslant 2$, and $||\tilde{S}_1|| \leqslant 2$ and that satisfy \eqref{funda1}, \eqref{funda11}, \eqref{Commutative 2(i)}, \eqref{Commutative 2(ii)} for some $G_1, 2G_2, 2\tilde{G}_1, \tilde{G}_2 \in \mathcal{B}(\mathcal{D}_{S_3})$ with $w(G_1 + \tilde{G}^*_2z) \leqslant 1$ and $w(G_2 + \tilde{G}^*_1z) \leqslant 1$ for all $z \in \mathbb{T}$. Then $\textbf{S}$ is a $\Gamma_{E(3; 2; 1, 2)}$-contraction.
	\end{thm}
	
	We end this section by a concluding remark on Theorem \ref{Schaffer Model 1} and Theorem \ref{Schaffer Model 2}.
	
	\begin{rem}\label{Rem 3}
		An important observation is that Theorem \ref{Schaffer Model 1} is nothing but the conditional isometric dilation of $\Gamma_{E(3; 3; 1, 1, 1)}$ described in [Theorem $4.6$, \cite{apal2}]. Here we have reformulated the conditional dilation in a model theoretic point of view. On the other hand, Theorem \ref{Schaffer Model 2} is also a model theoretic reformulation of [Theorem $4.7$, \cite{apal2}] where the conditional isometric dilation of $\Gamma_{E(3; 2; 1, 2)}$-contraction is developed.
	\end{rem}

	\vspace{.5cm}
	
	\noindent (A. Pal) \sc{Department of Mathematics, IIT Bhilai, 6th Lane Road, Jevra, Chhattisgarh 491002}\\
	{E-mail address:} {A. Pal:avijit@iitbhilai.ac.in}

	\noindent (B. Paul) \sc{Department of Mathematics, IIT Bhilai, 6th Lane Road, Jevra, Chhattisgarh 491002}\\
	{E-mail address:} {B. Paul:bhaskarpaul@iitbhilai.ac.in}  
	
\end{document}